\documentclass[a4paper, 11pt]{article} 

\usepackage[left=3cm,top=3cm,right=3cm,bottom=3cm]{geometry}
\usepackage[english]{babel}
\usepackage{microtype}
\usepackage{booktabs}
\usepackage{amsmath} 
\usepackage{amssymb,amsthm} 
\usepackage{floatrow}
\usepackage{graphicx}
\usepackage{subcaption}
\usepackage[mode=buildnew]{standalone}
\usepackage{todonotes}
\usepackage{hyperref}

\newtheorem{rem}{Remark}

\newtheorem{proposition}{Proposition}
\usepackage{tikz}
\usetikzlibrary{arrows.meta, positioning, calc}
\definecolor{mainblue}{RGB}{31, 119, 180}
\definecolor{reinfectionorange}{RGB}{214, 123, 25}

\DeclareFloatFont{scriptsize}{\scriptsize}
\usepackage{float}

\makeatletter
\newcommand\notsotiny{\@setfontsize\notsotiny\@vipt\@viipt}
\makeatother
\usepackage{xargs}   
\usepackage[pdftex,dvipsnames]{xcolor} 
\newcommandx{\agus}[2][1=]{\todo[linecolor=Plum,backgroundcolor=Plum!25,bordercolor=Plum,#1]{#2}}
\newcommandx{\marce}[2][1=]{\todo[linecolor=lime,backgroundcolor=lime!25,bordercolor=lime,#1]{#2}}

\allowdisplaybreaks 

\title{\LARGE \bf
	Mathematical modeling and control of Tuberculosis: Backward and Hopf bifurcations under case-finding and case-holding interventions
}

\author{A. Schaberger, M. Actis,  J. C. Bossio, A. H. González and A. D' Jorge 
	\thanks{ A. Schaberger, A. H. González, and A. D'Jorge are affiliated with the Institute of Technological Development for the Chemical Industry (INTEC), CONICET-UNL, Santa Fe, Argentina (e-mails: andresschaberger21@gmail.com, alejgon@santafe-conicet.gov.ar and agustinadj@santafe-conicet.gov.ar, respectively.)}
	\thanks{M. Actis is affiliated with the Faculty of Chemical Engineering (FIQ) - CONICET-UNL, Santa Fe, Argentina 
		(e-mail: marceactis@gmail.com)}%
	\thanks{J. C. Bossio is affiliated with the National Institute of Respiratory Diseases (INER), Santa Fe, Argentina 
		(e-mail: jcbossio29@gmail.com)}%
}

\begin{document}
	
	\maketitle
	
	\begin{abstract}
		Reinfection and relapse are known to induce backward bifurcation (a subcritical pitchfork bifurcation) in many infectious disease models, particularly tuberculosis. This phenomenon implies the coexistence of multiple equilibria when the basic reproduction number is below one: a stable disease-free equilibrium and two endemic equilibria, one stable and one unstable. As a result, hysteresis may arise, meaning that reducing the reproduction number below one may not eliminate the disease given that the system remains in the basin of attraction of the endemic state.
		
		Another relevant phenomenon, often associated with intervention delays, is the occurrence of supercritical Hopf bifurcations. In this case, the endemic equilibrium loses stability for reproduction numbers above one, leading to sustained oscillations in the form of stable limit cycles. Although less explored in the tuberculosis literature, this behavior has important implications for disease control.
		
		In this work, we analyze a tuberculosis model to investigate these dynamical effects and show that Hopf bifurcation can arise even without delays or saturation effects. The bifurcation analysis focuses on two key public health interventions---case finding and case holding---and their combined effect on disease dynamics. Numerical simulations, based on epidemiological data from Salta, Argentina---a province with high TB burden--- illustrate the challenges these dynamics pose for the design and implementation of public health policies.
	\end{abstract}

\section{Introduction}

\subsection*{The problem of Tuberculosis}

Tuberculosis (TB) is an infectious disease caused by the bacterium \textit{Mycobacterium tuberculosis}. Transmission occurs primarily through inhalation of airborne particles released when an individual with active TB coughs, sneezes, speaks, or otherwise expels respiratory secretions. Although TB is preventable and curable, at least in principle \cite{walzl2011immunological}, it remains the leading cause of death from a single infectious disease worldwide and is among the ten leading causes of death globally \cite{WHOTuberculosis2025}.

Following infection, individuals may enter a latent or an active stage. Latently infected individuals are generally asymptomatic and non-infectious, whereas individuals with active TB may transmit the disease. Approximately 90\% of infected individuals remain in the latent stage, and it is estimated that nearly one-quarter of the world's population carries latent \textit{M. tuberculosis} infection, constituting a large reservoir for future active cases \cite{houben2016global}. Although the duration of latency is highly variable, approximately 5--15\% of infected individuals are expected to develop active TB, with the risk being particularly high during the first 2--5 years following infection. This risk increases substantially in the presence of comorbidities and adverse conditions, including HIV infection, diabetes, COVID-19, kidney failure, malnutrition, overcrowding, and substance abuse \cite{Lawn2011}. Socioeconomic conditions, and poverty in particular, further contribute to TB transmission and disease progression \cite{paicanadian,BoletinN7TBArg}. Malnutrition \cite{lonnroth2010consistent}, alcohol abuse \cite{rehm2009association}, and drug addiction \cite{deiss2009tuberculosis,WHOTuberculosis2024} are additional factors associated with an increased risk of disease reactivation.

For drug-sensitive TB, standard first-line treatment lasts approximately six months and can achieve success rates of up to 95\%, with relapse rates below 5\% within two years after treatment completion. Treatment adherence is therefore essential. Interruptions or inadequate treatment may promote the emergence of drug-resistant strains, which generally require longer, more expensive, and less effective second-line therapies \cite{WHOTuberculosis_DrugResistance}. Moreover, the circulation of resistant strains increases the risk of secondary infections through exogenous reinfection \cite{chiang2005exogenous}, thereby complicating TB diagnosis, treatment, and control.

Despite substantial progress in prevention and treatment, TB remains a major global health challenge affecting populations across regions and age groups. In response, the World Health Organization (WHO) launched the \textit{End TB Strategy} in 2015 \cite{world2015end}, with the goal of ending the global TB epidemic through integrated patient-centered care, supportive policies and health systems, and intensified research and innovation. The elimination of TB is also explicitly included among the targets of the United Nations Sustainable Development Goals (SDGs).

\subsection*{TB modeling and control}

Mathematical models of infectious diseases provide valuable tools for investigating transmission mechanisms, identifying epidemiological thresholds, and assessing the potential impact of public health interventions. Their predictive and explanatory capabilities, however, depend critically on the quality and availability of epidemiological data. Among the most widely used frameworks for TB modeling are compartmental models derived from the classical SEIR structure \cite{kermack1927,brauer2019mathematical}, in which the population is divided into susceptible ($S$), exposed or latent ($E$), infectious ($I$), and removed ($R$) compartments. The resulting dynamics are governed by epidemiologically meaningful parameters, including the transmission rate, the rate of progression from latency to active disease, and the rate of removal from the infectious class.

A wide variety of extensions of this basic framework have been proposed to account for relevant features of TB epidemiology. These include vaccination, early and persistent latency, diagnosed and undiagnosed infections, treatment, comorbidities, fast and slow disease progression, endogenous reinfection (relapse), exogenous reinfection, and the emergence of drug resistance \cite{silva2015optimal}. Despite the substantial differences among these formulations, several qualitative dynamical features recur across TB models. In particular, both backward and Hopf bifurcations have been reported. Backward bifurcations are commonly associated with mechanisms such as reinfection and relapse \cite{wangari2018backward}, whereas Hopf bifurcations may arise in connection with delays in diagnosis or treatment \cite{das2025exploring,zhang2023dynamical}.

For instance, \cite{liu2011global} investigates an SVLIT model, comprising susceptible, vaccinated, latent, infectious, and treated populations, and establishes the occurrence of a forward bifurcation at a basic reproduction number equal to one. The analysis also demonstrates the potential effectiveness of vaccination and treatment under appropriate conditions, although delays are not considered. More recently, \cite{Peter2025model} introduced a TB model incorporating vaccination, treatment failure, and drug resistance. Their results emphasize that reducing the basic reproduction number below one is essential for disease eradication and highlight the role of vaccination and effective treatment in achieving this objective.

Several studies have explicitly incorporated intervention strategies into TB models. In \cite{gao2018optimal}, the SVLIT framework is extended by introducing two control mechanisms: case finding, aimed at identifying latent infections, and case holding, aimed at treating active cases. The resulting optimal control problem seeks to minimize the numbers of latent, infectious, and treated individuals while accounting for the cost of the interventions. The analysis suggests that the simultaneous application of the considered control measures provides the greatest benefit when delays are absent.

Similarly, \cite{silva2016optimal} considers an $SL_1IL_2R$ model in which a diagnostic delay affects the infectious compartment. The authors derive conditions for the occurrence of a forward bifurcation and formulate an optimal control strategy involving early detection ($u_1$) and treatment of persistently latent individuals ($u_2$). Their results indicate that diagnostic delays can substantially reduce the effectiveness of control interventions.

An important feature of TB dynamics is the possibility of \textbf{backward bifurcation}. Several studies \cite{feng2000model,brauer2019mathematical,wangari2018backward} have shown that mechanisms such as reinfection and relapse can lead to the coexistence of multiple equilibria when the basic reproduction number is below one. In this situation, a stable disease-free equilibrium may coexist with two endemic equilibria, typically one unstable and one stable. Consequently, the condition $\mathcal{R}_0<1$, which is sufficient for disease eradication in models exhibiting a forward bifurcation, may no longer be sufficient in the presence of a backward bifurcation.

Other extensions have demonstrated that nontrivial temporal dynamics may emerge from mechanisms such as time delays and nonlinear incidence. For example, \cite{zhang2023dynamical} considers delays in the latent stage and shows that sufficiently large delays can generate oscillatory behavior through \textbf{Hopf bifurcations}. Depending on the system parameters, stable or unstable periodic solutions may arise as the delay crosses a critical threshold. Similar oscillatory or multistable behavior can result from saturation in the infection incidence \cite{das2024exploring}, leading to the formation of so-called endemic bubbles with potentially relevant epidemiological consequences.

\subsection*{Contribution}

In this work, we investigate a TB model that exhibits both backward and Hopf bifurcations in the absence of either time delays or saturation effects. We provide a detailed analysis of the resulting dynamical behavior, with particular emphasis on the bifurcation structure induced by variations in key epidemiological and intervention parameters. A central aspect of the analysis is the explicit distinction between two major components of TB control: \textit{case finding}, associated with the detection of new cases, and \textit{case holding}, associated with the effective treatment and retention of infected individuals under care. We analyze the influence of the corresponding parameters on the equilibrium structure and bifurcation behavior of the model.

Finally, numerical simulations based on epidemiological data from the province of Salta, Argentina, are used to illustrate the theoretical results and to assess the implications of the identified dynamical phenomena for TB control. The results provide insight into how the interaction between case-finding and case-holding interventions may affect disease persistence and the effectiveness of strategies aimed at reducing TB transmission.

\section{Model description}

From a dynamical perspective, TB exhibits particular features due to feedback mechanisms within the infected compartments. Individuals who are already infected—either in the latent state or effectively treated—may be reinfected by a different (possibly resistant) strain, while treated individuals may also relapse without reinfection\footnote{Reinfection and relapse are often referred to as exogenous and endogenous reinfection, respectively.}.

A wide range of compartmental models has been proposed to describe TB dynamics, from simple frameworks focused on stability and bifurcation analysis to more detailed models tailored to specific regions \cite{vesga2022prioritising}. Here, we consider a variation of the model introduced in~\cite{feng2000model} and~\cite{castillo2004dynamical} (Section 4.5), designed to capture the main nonstandard features of TB dynamics\footnote{The difference between model~\eqref{eq:model} and those in~\cite{feng2000model} and~\cite{castillo2004dynamical} is the inclusion of a fast reinfection term, $cq\beta I(t)E(t)$, described next.}:
\begin{eqnarray}\label{eq:model}
		\dot{U}(t) &=& b(N) \!-\! \beta I(t) U(t) \!-\! \mu U(t),\\
		\dot{L}(t) &=& (1-p)\beta I(t) U(t) + c(1-q) \beta I(t) E(t) \!-\! (v + \mu) L(t)-k \beta I(t)L(t), \\
		\dot{I}(t) &=& v L(t) \!+\! p \beta I(t) U(t) \!+\! \rho E(t) +cq\beta I(t)E(t) \!-\! (\phi \!+\! \mu_T \!+\! \mu) I(t) \!+\! k \beta I(t)L(t), \\
		\dot{E}(t) &=& \phi I(t) - c \beta I(t) E(t) - (\rho + \mu) E(t), 
\end{eqnarray}

where $N = U(t)+L(t)+I(t)+E(t)$. The state vector is given by $x(t) = [U(t), L(t), I(t), E(t)]$, and the parameter set can be arranged in the vector $p_r = \{b, \beta, \mu, \mu_T, v, \phi, p, q, \rho, c, k\}$ (see Tables~\ref{tab:sysvar}~and~\ref{tab:syspar}). We assume that disease-induced mortality $\mu_T I(t)$ is negligible compared to the total natural mortality ($\mu_T I(t) \ll \mu N$), which is a typical simplification in this kind of model. Furthermore, we select $b(N):=\mu N$, so that the total population fulfill
\begin{eqnarray*}
    \dot N=\dot U+\dot L+\dot I+\dot E= -\mu_T I,
\end{eqnarray*}
which considering that $N>>I$, for all possible values of $I$, allows us to consider $N$ approximately constant in the time window of analysis. This assumption properly represents the reality of places with a relative small TB prevalence, in which the TB mortality does not affect the population size in the middle term.

Susceptible individuals ($U$) become infected at rate $\beta I U$, entering either the latent class $L$ (with proportion $1-p$) or directly the infectious class $I$ (with proportion $p$). Latent individuals progress to active infection at a rate $v$, but can also be reinfected (rate $k\beta I L$), providing an additional pathway to $I$. Infectious individuals are treated at rate $\phi$, moving to the recovered class $E$.

Recovered individuals can relapse (rate $\rho E$) or be reinfected: fast reinfection leads directly to $I$ (rate $cq\beta I E$), while slow reinfection returns individuals to latency (rate $c(1-q)\beta I E$). The parameter $\beta$ represents the transmission intensity, while $c\beta$ and $k\beta$ account for the reinfection among previously infected individuals. Parameters $p$ and $q$ determine the fractions of individuals who bypass latency after primary infection and reinfection, respectively.
As the system in \cite{feng2000model}, system~\eqref{eq:model} is positive and the sum of the states does not overpass $N$. In other words, the set $\mathbb X:=\{x \in \mathbb R^4: U\geq 0,~L\geq 0,~I\geq 0,~E\geq 0,~U+L+I+E\leq N\}$ is invariant. The system structure is shown in Figure~\ref{fig:basic_scheme}.

\begin{figure}[ht]
	\centering
	\includegraphics[width=0.9\columnwidth]{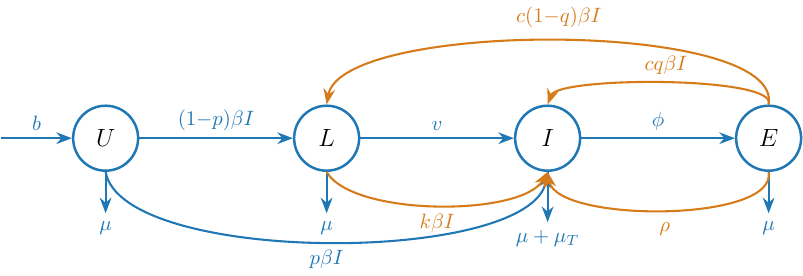}
	\caption{Proposed dynamical model structure. In red are plotted the reinfection and relapse.}
	 \label{fig:basic_scheme}
\end{figure}

\begin{table}[ht]
\small
\caption{Model state variables and their epidemiological interpretation.}
\centering
\begin{tabular}{@{} c p{9cm} c @{}}
\toprule
\textbf{Variable} & \textbf{Epidemiological class} & \textbf{Units} \\
\midrule
$U$ & Susceptible (uninfected) & individuals \\[4pt]
$L$ & Latent (infected but non-infectious; immune-stabilized) & individuals \\[4pt]
$I$ & Active TB (infectious) & individuals \\[4pt]
$E$ & Effectively treated / recovered (clinical cure) & individuals \\
\bottomrule
\end{tabular}
\label{tab:sysvar}
\end{table}

\subsection{Model reduction} \label{sec:modered}

Model~\eqref{eq:model} describes the population dynamics through the compartments and the flows connecting them. Since the total population is assumed to remain constant, i.e.,
\begin{eqnarray}
	N=U(t)+L(t)+I(t)+E(t), \qquad \forall t\geq 0,
\end{eqnarray}
the latent compartment can be eliminated by writing
\begin{eqnarray}
	L(t)=N-U(t)-I(t)-E(t).
\end{eqnarray}
Substituting this expression into the equations for $\dot U$, $\dot I$, and $\dot E$ yields the following third-order model:
\begin{subequations}\label{eq:redmodel}
	\begin{align}
		\dot{U}(t) &= \mu N - \mu U(t) - \beta I(t) U(t) ,\\
		\dot{I}(t) &= vN -v U(t) + (\rho-v) E(t) + (k N \beta - \phi - \mu_T - \mu - v) I(t)  \nonumber \\
		&\quad + (p-k) \beta I(t) U(t) + (cq-k)\beta I(t)E(t) - k \beta I^2(t), \\
		\dot{E}(t) &= \phi I(t) - (\rho + \mu) E(t) - c \beta I(t) E(t),
	\end{align}
\end{subequations}
The eliminated state can subsequently be recovered as
\begin{eqnarray}
L(t)=N-U(t)-I(t)-E(t).
\end{eqnarray}
Accordingly, the state vector of the reduced system is
\begin{eqnarray}
x(t)=[U(t),I(t),E(t)]^\top.
\end{eqnarray}

To investigate the dynamical behavior of the model, a nominal parameter set must first be established. This requires parameter values that are either supported by reliable epidemiological data or justified by previously reported estimates in the literature. In this study, the parameter vector $p_r$ was estimated using epidemiological data from the province of Salta, Argentina, provided by the National Institute of Respiratory Diseases of Argentina (INER) \cite{BoletinN7TBArg}. INER is the national institution responsible for TB surveillance in Argentina and for reporting epidemiological information to the WHO.

The unknown parameters were estimated using a Markov Chain Monte Carlo (MCMC) approach \cite{Hastings1970,gelman2013bayesian}. In particular, the Metropolis--Hastings algorithm was employed to generate samples from the posterior parameter distribution. Starting from an initial parameter vector, candidate values are sequentially proposed and either accepted or rejected according to the corresponding acceptance probability. After convergence, the resulting samples provide an approximation of the target parameter distribution and allow the associated parameter estimates to be obtained.

In the estimation procedure, the parameters $N$, $b$, $\mu$, $\mu_T$, $p$, and $q$ were fixed according to values reported in the literature and available a priori information from INER. The reinfection coefficient from the latent compartment, $k$, was also fixed at $k=0.26$, based on previous studies addressing the role of reinfection in the emergence of backward bifurcations in TB models \cite{wangari2018backward}. This parameter plays a particularly important role in the bifurcation analysis presented below. The remaining parameters were estimated from the available epidemiological data. The resulting parameter values are reported in the last column of Table~\ref{tab:syspar}.

\begin{table}[ht]
\small
\caption{Model parameters: notation, biological interpretation, and estimated values for the Salta, Argentina case study. Obtained from data base (DB), estimated by identification (E), and taken from literature (L).}
\centering
\begin{tabular}{p{1cm} p{3cm} p{4.5cm} p{2cm}}
\toprule
\textbf{Param.} & \textbf{Name and units} & \textbf{Biological interpretation} & \textbf{Value} \\
\midrule
$N$ & Total population $[\footnotesize{\text{ind}}]$ & Constant total population under study & $1{,}424{,}397$~(DB) \\[4pt]
$b$ & Recruitment rate $\left[\frac{\text{ind}}{\text{yr}}\right]$ & Annual inflow of newborns into the susceptible class & $18{,}261$~(DB) \\[4pt]
$\beta$ & TB transmission coefficient $\left[\frac{1}{\text{ind}\cdot\text{yr}}\right]$ & $\beta I U$: rate at which susceptible individuals ($U$) become infected through contact with active TB cases ($I$) & $1.3\!\times\!10^{-5}$~(E)\\[4pt]
$\mu$ & Natural per capita death rate $\left[\frac{1}{\text{yr}}\right]$ & $\mu X$: natural mortality outflow from any compartment $X$ & $1/78$~(DB) \\[4pt]
$\mu_T$ & TB-induced per capita death rate $\left[\frac{1}{\text{yr}}\right]$ & $\mu_T I$: excess mortality due to active TB & $0.1513$~(DB) \\[4pt]
$v$ & Per capita slow-progression rate $\left[\frac{1}{\text{yr}}\right]$ & $v L$: rate at which latent individuals progress to active disease & $1.7\!\times\!10^{-3}$~(E) \\[4pt]
$\phi$ & Per capita treatment rate $\left[\frac{1}{\text{yr}}\right]$ & $\phi I$: rate at which active cases are effectively treated (case holding) & $2.88$~(E) \\[4pt]
$p$ & Fast-progression proportion $[-]$ & Fraction of newly infected individuals who progress rapidly to active disease; $(1-p)$ enter the latent class & $0.05$~(L~\cite{wangari2018backward}) \\[4pt]
$\rho$ & Per capita relapse rate $\left[\frac{1}{\text{yr}}\right]$ & $\rho E$: rate at which treated individuals relapse into the active class & $2.4\!\times\!10^{-4}$~(E) \\[4pt]
$c$ & Exogenous reinfection coefficient $[-]$ & Scales the susceptibility of treated individuals to reinfection; $c\beta I E$: reinfection flow from the treated class & $0.084$~(E) \\[4pt]
$q$ & Fast-reinfection proportion $[-]$ & Fraction of reinfected treated individuals who progress rapidly; $(1-q)$ become latent & $0.05$~(L~\cite{wangari2018backward}) \\[4pt]
$k$ & Latent reinfection coefficient $[-]$ & $k\beta I L$: rate at which latent individuals become active through reinfection by a new strain & $0.26$~(L~\cite{wangari2018backward}) \\
\bottomrule
\end{tabular}
\label{tab:syspar}
\end{table}

\subsection{Equilibria}\label{sec:equilibria}

To characterize the system, its equilibria and their stability must be determined. By zeroing the equations in \eqref{eq:redmodel}, we obtain
\begin{subequations}\label{eq:equil1}
	\begin{align}
		\mu N - \mu U - \beta I_e U_e &= 0 ,\\
		vN -v U_e + (\rho-v) E_e+ \delta I_e + (p-k) \beta I_e U_e + (cq-k)\beta I_e E_e - k \beta I_e^2 &= 0 \\
	    \phi I - (\rho + \mu) E_e - c \beta I_e E_e &= 0.
	\end{align}
\end{subequations}
From (\ref{eq:equil1}a) and (\ref{eq:equil1}c), we can express $U_e$ and $E_e$ as functions of $I_e$,
\begin{eqnarray}\label{eq:Ue}
    U_e = \frac{\mu N}{\mu+\beta I_e},
\end{eqnarray}
\begin{eqnarray}\label{eq:Ee}
    E_e = \frac{\phi I_e}{\mu+\rho+ c\beta I_e},
\end{eqnarray}
which, replaced in (\ref{eq:equil1}b), gives:
\begin{eqnarray}\label{eq:Ie}
    vN - \frac{v \mu N}{\mu+\beta I_e} + \frac{(\rho-v)\phi I_e}{\mu+\rho+ c\beta I_e} + \delta I_e + \frac{\mu N (p-k) \beta I_e}{\mu+\beta I_e} + \frac{(cq-k)\beta \phi I_e^2}{\mu+\rho+ c\beta I_e} - k \beta I_e^2 = 0.
\end{eqnarray}
This later equation can be simplified as
\begin{eqnarray}\label{eq:Ie1}
     w_1 I_e^4 + w_2 I_e^3 + w_3 I_e^2 + w_4 I_e= 0,
\end{eqnarray}
where coefficients $w_1,w_2,w_3,w_4$ read
\begin{align} \label{eq:coeficientes_w}
    w_1&= -ck\beta^3, \nonumber \\
    w_2&=N\beta ^3ck-\beta ^2c\mu_T-\beta ^2c\phi -\beta ^2k\mu -\beta ^2cv-\beta ^2k\phi -\beta ^2k\rho -\beta ^2\,c\,\mu \nonumber\\
    &\quad-\beta ^2ck\mu +\beta ^2cq\phi, \nonumber \\
    w_3&= N\,\beta ^2\,k\,\mu -\beta \,\mu \,\mu_T-\beta \,\mu \,\phi -\beta \,\mu \,\rho -\beta \,\mu_T\,\rho -\beta \,\mu \,v-\beta \,\phi \,v-\beta \,\rho \,v-\beta \,c\,\mu^2 \nonumber \\
    &\quad-\beta \,k\,\mu ^2-\beta \,c\,\mu \,\phi -\beta \,c\,\mu \,v-\beta \,k\,\mu \,\phi -\beta \,k\,\mu \,\rho -\beta \,\mu ^2+N\,\beta ^2\,c\,v+N\,\beta ^2\,k\,\rho  \nonumber \\
    &\quad -\beta \,c\,\mu \,\mu_T+ \beta \,c\,\mu \,\phi \,q+N\,\beta ^2\,c\,\mu \,p, \nonumber \\
    w_4&=N\,\beta \,\mu ^2\,p-\mu ^2\,\phi -\mu ^2\,\rho -\mu ^2\,v-\mu ^3-\mu \,\mu_T\,\rho -\mu \,\phi \,v-\mu \,\rho \,v-\mu ^2\,\mu_T+N\,\beta \,\mu \,v \nonumber \\
    &\quad +N\,\beta \,\rho \,v+N\,\beta \,\mu \,p\,\rho.
\end{align}

Now, the four roots of the quartic polynomial \eqref{eq:Ie1}, $I_e^{(i)}$, with $i\in \{1,2,3,4\}$, together with equations \eqref{eq:Ue} and \eqref{eq:Ee}, are candidates to form the equilibria of the system:
\begin{eqnarray*}
x_{e}^{(i)}=[U_{e}^{(i)},I_{e}^{(i)},E_{e}^{(i)}].
\end{eqnarray*}

However, for a triplet $[U_{e}^{(i)},I_{e}^{(i)},E_{e}^{(i)}]$ to be an equilibrium with epidemiological meaning, each component must be non-negative and real (nnr); i.e., it makes no sense a negative or complex number of individuals. By \eqref{eq:Ue} and \eqref{eq:Ee} (and the fact that all the parameters are positive), each component of the triplet $[U_{e}^{(i)},I_{e}^{(i)},E_{e}^{(i)}]$ will be nnr if and only if the corresponding component $I_{e}^{(i)}$ is nnr.

The first equilibrium candidate is the disease-free equilibrium,
\begin{eqnarray}
DFE:=x_{e}^{(1)}=[U_{e}^{(1)},I_{e}^{(1)},E_{e}^{(1)}]=[N,0,0],
\end{eqnarray}
which is obtained from the root $I_e^{(1)}=0$, and has non-negative and real components for all positive parameter values. The other three candidates, denoted without loss of generality as endemic equilibrium candidates, are given by
\begin{eqnarray*}
EE_1:= x_{e}^{(2)}=[U_{e}^{(2)},I_{e}^{(2)},E_{e}^{(2)}],\\
EE_2:= x_{e}^{(3)}=[U_{e}^{(3)},I_{e}^{(3)},E_{e}^{(3)}],\\
EE_3:= x_{e}^{(4)}=[U_{e}^{(4)},I_{e}^{(4)},E_{e}^{(4)}].
\end{eqnarray*}

Due to the large number of parameters, analyzing the candidates $EE_1$, $EE_2$, and $EE_3$ for all possible combinations of parameters is not feasible. Therefore, we focus on studying their behavior as the critical parameters associated with external interventions vary within reasonable ranges, while keeping the remaining parameters fixed according to reliable datasets or well-justified assumptions. The formal analysis of the behavior (existence and stability) of the $DFE$, $EE_1$, $EE_2$, and $EE_3$ as critical parameters vary is carried out through bifurcation analysis in Section~\ref{sec:bifu}.

\subsection{Stability}

The stability of the equilibria is determined by evaluating the Jacobian matrix of System~\eqref{eq:model},
\begin{eqnarray}\label{eq:jacob3D}
    J(x,p_r)\!\!=\!\!\begin{bmatrix}
    -\beta I \!\! - \!\! \mu & -\beta U &\!\! 0\\
     (p \!\!-\!\! k)\beta I \!\!-\!\!v& \delta \!\!+\!\! (p\!\!-\!\!k)\beta U \!\!+ \!\!(cq\!\!-\!\!k)\beta E \!\!-\!\!2k\beta I& (\rho\!\!-\!\!v) \!\!+\!\! (cq\!\!-\!\!k)\beta I\\
     0 \!\!&\!\! \phi\!\!-\!\!c \beta E  & -c\beta I\!\!-\!\!\rho\!\!-\!\!\mu
    \end{bmatrix},
\end{eqnarray}
at $DFE$, $EE_1$, $EE_2$, and $EE_3$. For each equilibrium, three eigenvalues are obtained, which determine its stability. Let $\lambda_j^{(i)}$ denote the $j$-th eigenvalue of $J(x_e^{(i)}, p_r)$ corresponding to the $i$-th equilibrium, where $j \in \{1,2,3\}$ and $i \in \{1,2,3,4\}$.

If all eigenvalues are negative real numbers, then $x_e^{(i)}$ is a stable node (stable without oscillations). If at least one eigenvalue is positive and real, then $x_e^{(i)}$ is unstable (without oscillations). If the eigenvalues are complex with negative real parts, then $x_e^{(i)}$ is a stable focus (stable with oscillations). Finally, if at least one eigenvalue has a positive real part and a nonzero imaginary part, then $x_e^{(i)}$ is an unstable focus (unstable with oscillations).

To numerically determine the equilibrium type and stability, we define two quantities:
(i) the maximal real part of the eigenvalues of the Jacobian evaluated at each equilibrium:
\begin{eqnarray}\label{eq:max_re_eigval}
\lambda^{(i)}_{\max} := \max_{j \in \mathbb I_3} (\mathrm{Re}(\lambda^{(i)}_j)), \quad i\in \mathbb I_{4},
\end{eqnarray}
and (ii) the imaginary part of the eigenvalues of the Jacobian evaluated at each equilibrium:
\begin{eqnarray}\label{eq:im_part}
\lambda^{(i)}_{j,\mathrm{Im}} := \mathrm{Im}(\lambda^{(i)}_j), \quad j\in \mathbb I_{3}, \quad i\in \mathbb I_{4}.
\end{eqnarray}
In this way, for the $i$-th equilibrium, if $\lambda^{(i)}_{\max}<0$, then $x_e^{(i)}$ is either a stable node or a stable focus. If all $\lambda^{(i)}_{j,\mathrm{Im}}$, $j \in \mathbb I_{3}$, are zero, then $x_e^{(i)}$ is a stable node; otherwise, if at least one of them is nonzero, it is a stable focus.
Conversely, if $\lambda^{(i)}_{\max}>0$, then $x_e^{(i)}$ is either an unstable node or an unstable focus\footnote{The case $\lambda^{(i)}_{\max}=0$ means that the equilibrium is non-hyperbolic, and linearization cannot be used to determine its stability. In any case, we are interested in how and when $\lambda^{(i)}_{\max}$ changes sign.}. If all $\lambda^{(i)}_{j,\mathrm{Im}}$, $j \in \mathbb I_{3}$, are zero, then $x_e^{(i)}$ is an unstable node; otherwise, it is an unstable focus.

Next, in Section~\ref{sec:bifu}, once a parameter estimation has been performed and the parameters associated with external interventions have been determined, the aforementioned eigenvalues, as well as the two quantities defined above, will be considered as functions of those parameters. Then, the existence, type, and stability of each equilibrium $x_{e}^{(1)}$, $x_{e}^{(2)}$, $x_{e}^{(3)}$, and $x_{e}^{(4)}$ will be analyzed in order to show the possible emergence of bifurcations.


\subsection{External interventions}\label{sec:extinterv}

A standard tool in epidemiological modeling is the \textit{basic reproduction number} $R_0$, defined as the average number of secondary infections produced by a single infectious individual in an entirely susceptible population. It is typically used to gauge the force of infection in a given scenario and can be computed via the next generation matrix method \cite{diekmann2010construction}. For our model, it takes the form
\begin{eqnarray}\label{eq:R0}
	R_0 (p_r) = \frac{p \beta b}{(\mu+v)(\mu+\mu_T+\phi)} + \frac{v \beta b}{\mu(\mu+v)(\mu+\mu_T+\phi)} ,
\end{eqnarray}
with nominal numerical value $R_0^{nom}=0.9647$ (using the parameters in Table~\ref{tab:syspar}).

In most epidemiological models, $R_0$ alone determines whether the infection persists ($R_0>1$) or dies out ($R_0<1$) \cite{castillo2004dynamical}. TB models, however, do not generally share this property: reinfection and relapse can give rise to a \textit{backward bifurcation} (a type of subcritical pitchfork bifurcation), through which the infection persists even when $R_0<1$. This behavior has been documented in several studies \cite{castillo2004dynamical,wangari2018backward,gerberry2016practical} and, as we show in Section~\ref{sec:bifu}, also arises in our nominal model~\eqref{eq:model}.

From a control standpoint, we are interested in how a health system can intervene to reduce prevalence $I$ to undetectable levels. Such interventions fall into two broad categories: (i) identification of cases together with chemoprophylaxis, aimed at preventing new infections, and (ii) treatment and adherence, aimed at accelerating the clinical cure of infected individuals \cite{blower1996control}. These are commonly referred to, respectively, as \textit{case finding} (identification and prophylaxis of latently infected individuals) and \textit{case holding} (treatment and adherence of individuals with active TB) \cite{gao2018optimal}. In our model, case finding is governed by the transmission coefficient $\beta$, and case holding by the treatment effectiveness rate per capita $\phi$.

This distinction motivates a departure from common practice. Rather than using $R_0$ as a single meta-parameter for control purposes --- as is customary, following conventions from other infectious diseases --- we propose working directly with $\beta$ and $\phi$. Two observations support this choice. First, $R_0$ (equation~\eqref{eq:R0}) collapses the distinct effects of $\beta$ and $\phi$ into a single number, which limits the ways in which the system can be steered or controlled. Second, because endemic persistence can occur even when $R_0<1$, driving the system below the threshold $R_0=1$ is not, by itself, an effective control objective.

\begin{rem}
	Treating $\beta$ and $\phi$ as separate manipulated variables lets the health system decide how to distribute effort between case finding and case holding, for instance by reducing intervention on one front to reinforce the other.
\end{rem}

Under this framework, System~\eqref{eq:redmodel} is regarded as the control system, with $\beta$ and $\phi$ acting as manipulated variables constrained to a set that is both reasonable and epidemiologically meaningful:
\[
\Omega := \{ (\beta,\phi) \in \mathbb{R}_{\geq 0}^2 : 0.5\times 10^{-5} \leq \beta \leq 2.5\times 10^{-5},\; 1.5 \leq \phi \leq 4 \},
\]
expressed in their corresponding units.

Because parameter estimation relied on data from a system that was already under intervention --- i.e., one in which the health system was actively reducing TB transmission and reinforcing treatment adherence within its available budget --- the nominal values of $\beta$ and $\phi$ lie in the interior of $\Omega$. The bounds of $\Omega$ were chosen so that the reproduction number ranges from $0.27$ to $3.44$ (against a nominal value of $R_0^{nom}=0.9647$), consistent with reported TB values worldwide \cite{ma2018quantifying}. With $\Omega$ defined, we next perform a bifurcation analysis to examine how the prevalence equilibria change as $\beta$ and $\phi$ vary over this set.

\section{Bifurcation Analysis}\label{sec:bifu}

\subsection{Existence of equilibria}\label{sec:existence}

Given the nominal fixed parameters obtained in Section~\ref{sec:modered} and the intervention-parameter set $\Omega$ introduced in Section~\ref{sec:extinterv}, we now determine which roots of equation~\eqref{eq:Ie1} correspond to meaningful equilibria $x_e^{(i)} = [U_e^{(i)}, I_e^{(i)}, E_e^{(i)}]$ --- that is, for which $(\beta,\phi)\in\Omega$ the roots $I_e^{(i)}$, $i \in \{1,2,3,4\}$, are non-negative and real (nnr). To lighten the notation, let $\tilde p_r$ denote the fixed parameters other than $\beta$ and $\phi$.

The root $I_e^{(1)}(\beta,\phi,\tilde p_r)= 0$ is trivially nnr throughout $\Omega$, so the disease-free equilibrium
\[
DFE=[U_e^{(1)}, I_e^{(1)}, E_e^{(1)}]=[N,0,0]
\]
exists for all $(\beta,\phi)\in\Omega$. The remaining roots, $I_e^{(2)}(\beta,\phi,\tilde p_r)$, $I_e^{(3)}(\beta,\phi,\tilde p_r)$, and $I_e^{(4)}(\beta,\phi,\tilde p_r)$, are nonzero by construction (since $w_4\neq 0$ in equation~\eqref{eq:Ie1}) and define the endemic equilibria $EE_1$, $EE_2$, and $EE_3$, respectively. These roots satisfy the cubic
\begin{eqnarray}\label{eq:Ie1_cubico}
	w_1 I_e^3 + w_2 I_e^2 + w_3 I_e + w_4 = 0,
\end{eqnarray}
whose roots are either all real, or one real and a complex-conjugate pair. Which of these two cases occurs is determined by the sign of the discriminant
\begin{eqnarray}\label{eq:det_exists}
	\Delta(\beta,\phi,\tilde p_r)
	&=&
	18 w_1 w_2 w_3 w_4
	- 4 w_2^3 w_4
	+ (w_2 w_3)^2
	- 4 w_1 w_3^3
	- 27 w_1^2 w_4^2,
\end{eqnarray}
which can in turn be written in terms of $\beta$, $\phi$, and $\tilde p_r$ as
\begin{eqnarray}\label{eq:delta0}
	\Delta(\beta,\phi,\tilde p_r)
	&=&
	a_4(\tilde p_r)\beta^6\phi^4
	+a_3(\tilde p_r)\beta^5\phi^3
	+a_2(\tilde p_r)\beta^4\phi^2
	+a_1(\tilde p_r)\beta^3\phi
	+a_0(\tilde p_r),
\end{eqnarray}
with coefficients $a_0(\tilde p_r),\ldots,a_4(\tilde p_r)$ given in the Appendix. Wherever the curve $\Delta(\beta,\phi,\tilde p_r)=0$ (Figure~\ref{fig:twoplots_a}, solid line) intersects $\Omega$, it splits the parameter set into two regions. Above the curve, where $\Delta<0$, the cubic has one real root and a complex-conjugate pair, so at most two equilibria can exist: the $DFE$ and, without loss of generality, the endemic equilibrium $EE_3$. Below the curve, where $\Delta>0$, the cubic has three distinct real roots, allowing for up to four equilibria: $DFE$, $EE_1$, $EE_2$, and $EE_3$\footnote{At $\Delta=0$ two real roots coincide, marking the transition from a complex-conjugate pair to two distinct real roots.}. In each region, the non-negativity of $EE_1$, $EE_2$, and $EE_3$ must be checked separately.

To determine the sign of the real roots, we use Descartes' rule of signs. For the cubic~\eqref{eq:Ie1_cubico} with $w_1\neq 0$, the number of negative real roots (with multiplicity) equals $N_s$ or $N_s-2$, where $N_s\in\{0,1,2,3\}$ is the number of sign changes in the sequence
\[
\{-w_1,\,w_2,\,-w_3,\,w_4\}.
\]
Similarly, the number of positive real roots equals $P_s$ or $P_s-2$, where $P_s$ counts the sign changes in $\{w_1,\,w_2,\,w_3,\,w_4\}$.

We use this to partition $\Omega$ according to the signs of $w_1$, $w_2$, $w_3$, and $w_4$ (their explicit expressions are given in \eqref{eq:coeficientes_w}). Over all of $\Omega$, $w_1(\beta,\phi,\tilde p_r)$ is negative. Figure~\ref{fig:twoplots_a} shows the curves $w_2=0$ (dash-dotted), $w_3=0$ (dashed), and $w_4=0$ (dotted), which further divide $\Omega$ into regions of constant sign; evaluating a test point in each region shows that $w_2$, $w_3$, and $w_4$ are negative above their respective curves and positive below them. We now examine each resulting region in turn.

\textbf{Region $\Omega_1$} (above $w_3=0$): here $w_2<0$, $w_3<0$, $w_4<0$, so the sequence $\{-w_1,\,w_2,\,-w_3,\,w_4\}$ changes sign three times, giving $N_s=3$ and hence one or three negative real roots. Since $\Delta<0$ in $\Omega_1$, only one root is real, and it must therefore be negative. Thus $DFE$ is the only equilibrium in $\Omega_1$.

\textbf{Region $\Omega_2$} (below $w_3=0$, above $\Delta=0$): here $w_2<0$, $w_3>0$, $w_4<0$, and $\Delta<0$. The sequence $\{-w_1,\,w_2,\,-w_3,\,w_4\}$ now changes sign only once, so the cubic has exactly one negative real root. As in $\Omega_1$, $DFE$ is the only equilibrium.

\textbf{Region $\Omega_3$} (above $w_4=0$, below $\Delta=0$): here $w_2<0$, $w_3>0$, $w_4<0$, and $\Delta>0$. The sequence $\{-w_1,\,w_2,\,-w_3,\,w_4\}$ again changes sign once, giving exactly one negative real root; since $\Delta>0$ guarantees three distinct real roots, the other two roots must be positive. In particular\footnote{The roots $I_e^{(2)}$ and $I_e^{(3)}$ are positive in this region by construction; below the curve $\Delta=0$ they satisfy $I_e^{(2)}\neq I_e^{(3)}$, and they coincide on the curve itself. We take $I_e^{(2)}$ to be the smaller positive root.},
\[
0=I_e^{(1)}<I_e^{(2)}<I_e^{(3)}.
\]
So $\Omega_3$ admits three meaningful equilibria: $DFE$, $EE_1$, and $EE_2$, while $EE_3$ has negative components.

\textbf{Region $\Omega_4$} (below $w_4=0$, above $w_2=0$): here $w_2<0$, $w_3>0$, $w_4>0$, and $\Delta>0$. The sequence $\{-w_1,\,w_2,\,-w_3,\,w_4\}$ now changes sign twice, leaving an ambiguous count of two or zero negative roots. We resolve this by counting positive roots instead, applying Descartes' rule directly to $\{w_1,\,w_2,\,w_3,\,w_4\}$. Since $w_1<0$, $w_2<0$, $w_3>0$, and $w_4>0$, this sequence has a single sign change, so there is exactly one positive real root. Because $\Delta>0$ ensures three distinct real roots and $w_4\neq 0$ rules out a zero root, the remaining two roots must be negative. Hence only $EE_2$ (the positive root) is meaningful, while $EE_1$ and $EE_3$ have negative components: $\Omega_4$ admits two equilibria, $DFE$ and $EE_2$.

\textbf{Region $\Omega_5$} (below $w_2=0$): here $w_2>0$, $w_3>0$, $w_4>0$, and $\Delta>0$. The same reasoning applies: since $w_1<0$ while $w_2,w_3,w_4>0$, the sequence $\{w_1,\,w_2,\,w_3,\,w_4\}$ changes sign once, giving exactly one positive real root, with the remaining two roots negative (as $\Delta>0$ and $w_4\neq 0$). As in $\Omega_4$, $\Omega_5$ admits two equilibria: $DFE$ and $EE_2$.

Figure~\ref{fig:twoplots_a} collects the curves $w_2=0$, $w_3=0$, $w_4=0$, and $\Delta=0$ together with the regions $\Omega_i$, $i=1,\dots,5$, while Figure~\ref{fig:twoplots_b} summarizes the corresponding regions of coexistence of one, two, and three epidemiologically meaningful equilibria.

\begin{figure}[htbp]
	\centering
	\begin{subfigure}{0.48\textwidth}
		\centering
		\includegraphics[width=\linewidth]{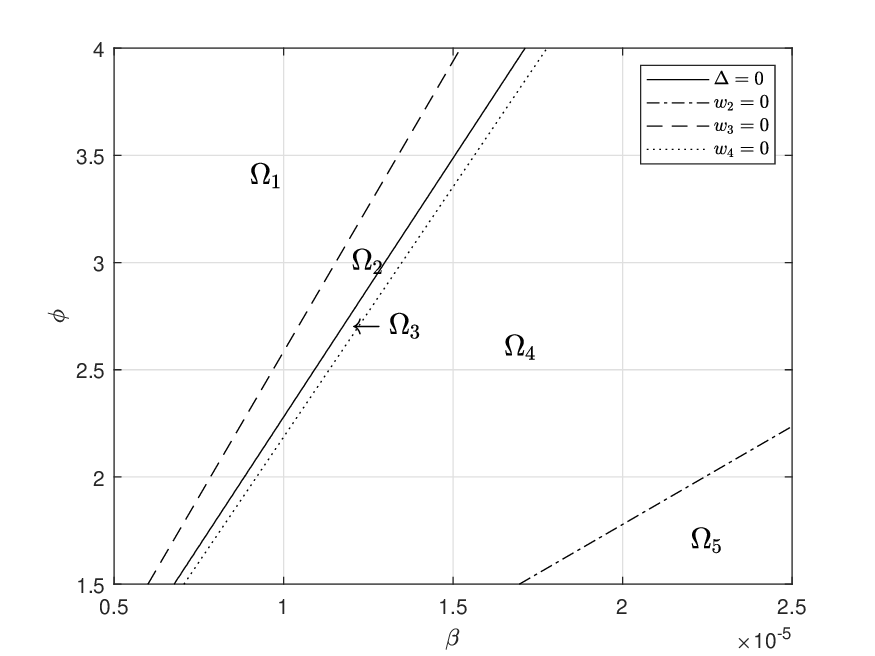}
		\caption{}
		\label{fig:twoplots_a}
	\end{subfigure}
	\hfill
	\begin{subfigure}{0.48\textwidth}
		\centering
		\includegraphics[width=\linewidth]{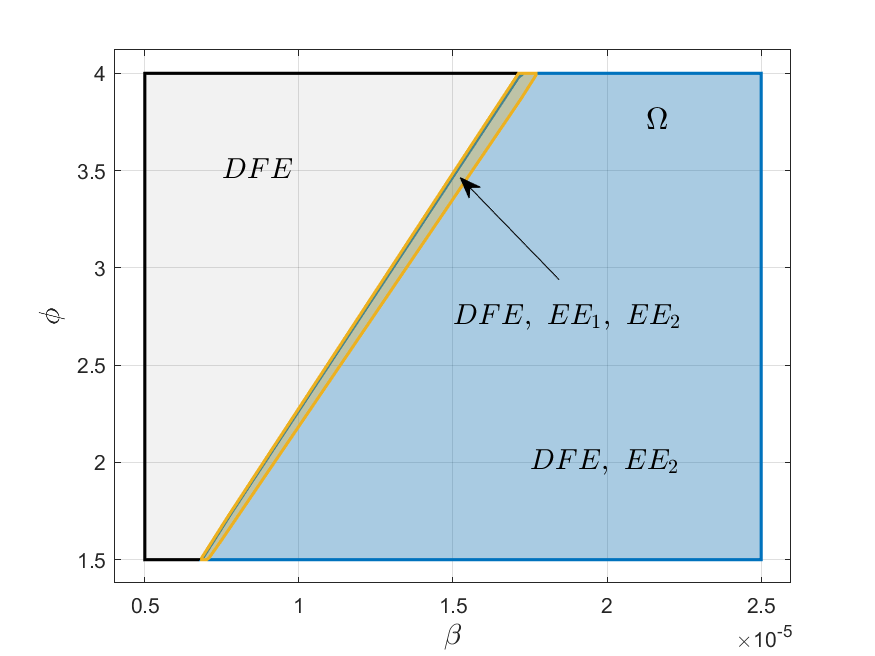}
		\caption{}
		\label{fig:twoplots_b}
	\end{subfigure}
	\caption{Existence. (a) Curves $w_2=0$ (point-dashed line), $w_3=0$ (dashed line), $w_4=0$ (dotted line), $\Delta=0$ (solid lines), and subsets $\Omega_1$, $\Omega_2$, $\Omega_3$, $\Omega_4$, and $\Omega_5$, in $\Omega$. (b) Regions in which $I_e^{(1)}$, $I_e^{(2)}$, and $I_e^{(3)}$ are non-negative and real in $\Omega$, leading to the existence of epidemiological equilibria $x_{e}^{(1)}=DFE$, $x_{e}^{(2)}=EE_1$, and $x_{e}^{(3)}=EE_2$, for System~\eqref{eq:redmodel}. There is a region where only $x_{e}^{(1)}=DFE$ is nnr, given by $\Omega_1 \cup \Omega_2$ (gray set), a region where $x_{e}^{(1)}=DFE$, $x_{e}^{(2)}=EE_1$ and $x_{e}^{(3)}=EE_2$ are nnr, given by $\Omega_3$ (yellow set), and a region where $x_{e}^{(1)}=DFE$ and $x_{e}^{(3)}=EE_2$ are nnr, given by $\Omega_4 \cup \Omega_5$ (blue set). $I_e^{(4)}$, and so $x_{e}^{(4)}=EE_3$, is never nnr in $\Omega$.}
	\label{fig:twoplots}
\end{figure}

\subsection{Stability of equilibria} \label{sec:stabilityeq}

To develop a formal bifurcation study, we next analyze the stability of the equilibria. To determine the stability of $DFE$, which exists for all $(\beta,\phi)\in \Omega$, we study the three eigenvalues of the Jacobian matrix $J(DFE,\beta,\phi,\tilde{p}_r)$. We see that they are all real in $\Omega$, but the maximal one, $\lambda^{(1)}_{\max}$, has a different sign in different regions (i.e., $DFE$ is a stable or an unstable node, depending on the values of $(\beta,\phi)\in \Omega$).
	
To formally determine the stability threshold of the $DFE$ in $\Omega$, we establish the curve where $\lambda^{(1)}_{\max}$ changes sign,
\begin{eqnarray}\label{eq:lambmax}
		\lambda^{(1)}_{\max}(\beta,\phi,\tilde{p}_r)=b_1(\tilde{p}_r) + b_2(\tilde{p}_r) \beta + b_3(\tilde{p}_r) \phi=0,
\end{eqnarray}
with $b_1(\tilde{p}_r):=-(\rho+\mu)\mu(\mu+\mu_T+v)$, $b_2(\tilde{p}_r):=N(\mu+\rho)(v+p\mu)$, and $b_3(\tilde{p}_r) :=\mu(\mu+v)$, while ensuring that
\begin{equation}\label{eq:DFEstabil}
		\lambda^{(1)}_{1,\mathrm{Im}}(\beta,\phi,\tilde{p}_r)
		=\lambda^{(1)}_{2,\mathrm{Im}}(\beta,\phi,\tilde{p}_r)
		=\lambda^{(1)}_{3,\mathrm{Im}}(\beta,\phi,\tilde{p}_r)
		=0,
\end{equation}
for all $(\beta,\phi)\in \Omega$. The curve~\eqref{eq:lambmax} coincides with the curve $w_4=0$, and divides the parameter space into two subsets: one where the $DFE$ is a stable node (above the curve) and one where it is an unstable node (below the curve). Figure~\ref{fig:threeplots_a} illustrates these regions, in green and red, respectively.
\begin{figure}[htbp]
\centering
    \begin{subfigure}{0.49\textwidth}
        \centering
        \includegraphics[width=\linewidth]{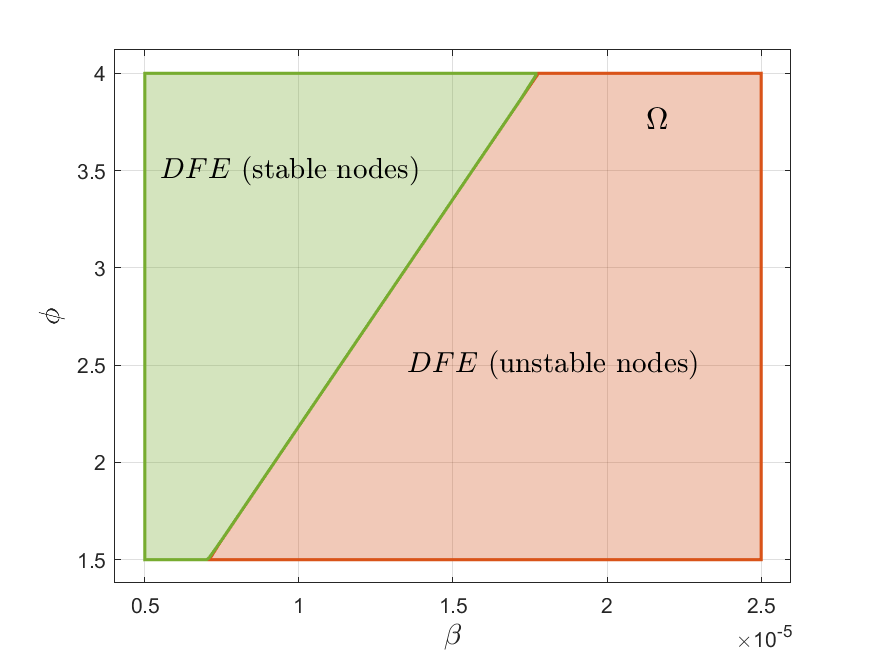}
        \caption{}
        \label{fig:threeplots_a}
    \end{subfigure}
    \hfill
    \begin{subfigure}{0.49\textwidth}
        \centering
        \includegraphics[width=\linewidth]{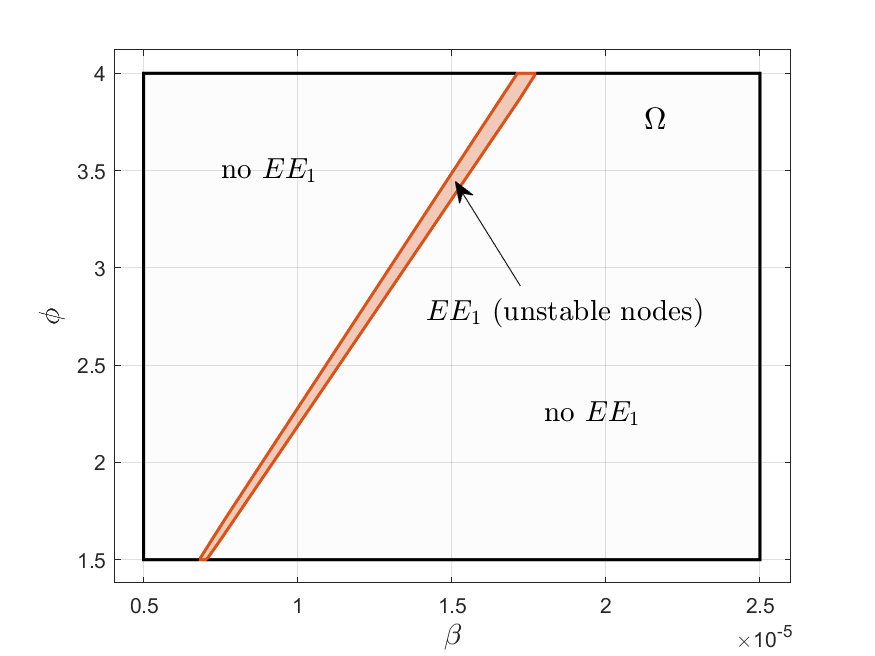}
        \caption{}
        \label{fig:threeplots_b}
    \end{subfigure}
    \\
    \begin{subfigure}{0.49\textwidth}
        \centering
        \includegraphics[width=\linewidth]{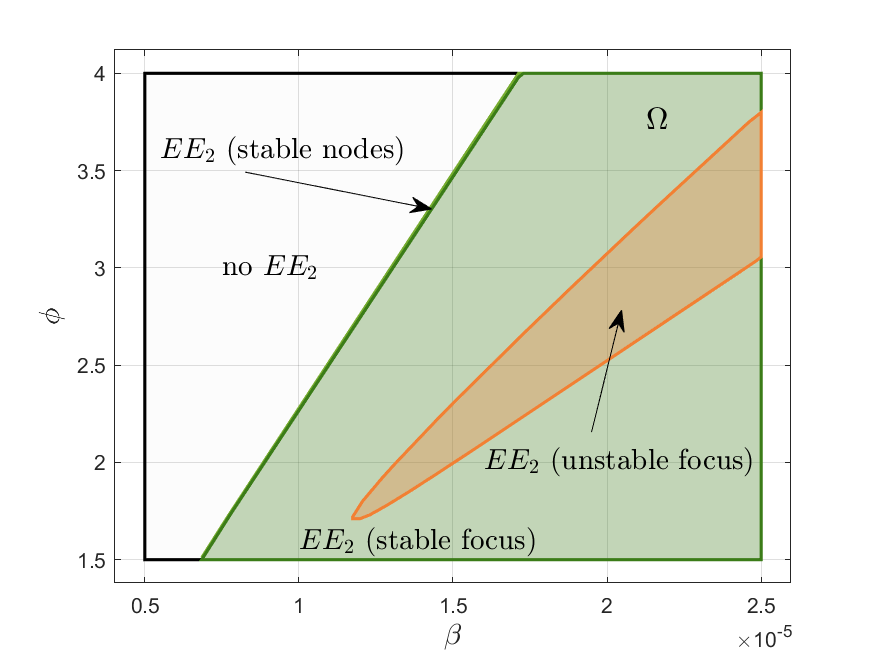}
        \caption{}
        \label{fig:threeplots_c}
    \end{subfigure}
    \caption{Stability. (a) Set $\Omega$, and the partitions in which the DFE is a stable node (light green subset) and an unstable node (light red subset). (b) Set $\Omega$, and the region where $EE_1$ exists and is an unstable node (light red set). The lower boundary coincides with that of the region where the DFE is stable. (c) Set $\Omega$, and the set where $EE_2$ exists and are stable nodes (dark green segment), stable focus (light green set), and unstable focus (dark red set).}
    \label{fig:threeplots}
\end{figure}
	
The equilibrium $EE_1$ exists only in a small range of values of $\beta$ and $\phi$ (given by the subset $\Omega_3$, as shown in Section~\ref{sec:existence}), and it is always an unstable node. This result is determined by analyzing the eigenvalues of the Jacobian $J(EE_1,\beta,\phi,\tilde{p}_r)$ and checking the conditions
\begin{equation}
	\lambda^{(2)}_{\max}(\beta,\phi,\tilde{p}_r) \geq 0,
\end{equation}
and
	\begin{equation}
		\lambda^{(2)}_{1,\mathrm{Im}}(\beta,\phi,\tilde{p}_r)
		=\lambda^{(2)}_{2,\mathrm{Im}}(\beta,\phi,\tilde{p}_r)
		=\lambda^{(2)}_{3,\mathrm{Im}}(\beta,\phi,\tilde{p}_r)
		=0.
	\end{equation}
Figure~\ref{fig:threeplots_b} shows the region where $EE_1$ exists and is an unstable node, in red. Its upper boundary is given by the curve $\Delta=0$, while the lower boundary coincides is given by $w_4=0$.
	
Finally, the equilibrium $EE_2$, which exists below the curve $\Delta=0$ (i.e., for large values of $\beta$ and small values of $\phi$) as shown in Section~\ref{sec:existence}, undergoes a sequence of stability changes as $\beta$ increases and $\phi$ decreases: it is first a stable node, then a stable focus, then an unstable focus, and finally a stable focus again.
	
The condition for $EE_2$ to be a stable node is given by the eigenvalues of the Jacobian $J(EE_2,\beta,\phi,\tilde{p}_r)$, i.e.,
\begin{equation}
	\lambda^{(3)}_{\max}(\beta,\phi,\tilde{p}_r) \leq 0,
\end{equation}
together with
\begin{equation}
	\lambda^{(3)}_{1,\mathrm{Im}}(\beta,\phi,\tilde{p}_r)
	=\lambda^{(3)}_{2,\mathrm{Im}}(\beta,\phi,\tilde{p}_r)
	=\lambda^{(3)}_{3,\mathrm{Im}}(\beta,\phi,\tilde{p}_r)
	=0.
\end{equation}
This defines a narrow region in $\Omega$, depicted in dark green in Figure~\ref{fig:threeplots_c}. Similarly, the condition for $EE_2$ to be a stable focus is given by
\begin{equation}
	\lambda^{(3)}_{\max}(\beta,\phi,\tilde{p}_r) \leq 0,
\end{equation}
provided that at least one of the following holds:
\begin{equation}
	\lambda^{(3)}_{i,\mathrm{Im}}(\beta,\phi,\tilde{p}_r)\neq 0,
	\quad i=1,2,3.
\end{equation}
This region is non-convex and is shown in light green in Figure~\ref{fig:threeplots_c}. Finally, the equilibrium $EE_2$ is an unstable focus when
\begin{equation}
	\lambda^{(3)}_{\max}(\beta,\phi,\tilde{p}_r) > 0,
\end{equation}
and at least one eigenvalue has a nonzero imaginary part:
\begin{equation}\label{eq:EE2unsfocus}
	\lambda^{(3)}_{i,\mathrm{Im}}(\beta,\phi,\tilde{p}_r)\neq 0,
	\quad i=1,2,3.
\end{equation}
This region is depicted in red in Figure~\ref{fig:threeplots_c}.

\begin{rem}
    The explicit expressions of conditions~\eqref{eq:DFEstabil}--\eqref{eq:EE2unsfocus} in terms of the system parameters are excessively cumbersome. Therefore, we present only the general conditions and verify them numerically.
\end{rem}

\subsubsection{A first picture of the three prevalence equilibria existence and stability under key parameters variations}

The proper bifurcation plot, which we leave for the final subsection, consists in a 3D plot showing the $I$-component of the three equilibria as a function of $\beta$ and $\phi$. However, it is far from trivial to draw conclusions of such a figure without a previous understanding of the system's equilibrium behavior when varying each parameter at a time (then, in Section~\ref{sec:control}, the formal complete bifurcation will be analyzed).

Figures~\ref{fig:bifu1_beta} shows $I_e^{(1)}$, $I_e^{(2)}$, and $I_e^{(3)}$ when $0.5\times 10^{-5} \leq \beta \leq 2.5\times 10^{-5}$  and $\phi$ is fixed at its nominal value, $\phi^{nom}$. For increasing values of $\beta$, we can identify six regions with different combination of existence and stability of the equilibria.
\begin{enumerate}
	\item For $0.5\times 10^{-5} \leq \beta \leq 1.25 \times 10^{-5}$, System~\eqref{eq:redmodel} has only the $DFE$ stable node,
	\item for $1.25\times 10^{-5} \leq \beta \leq 1.29 \times 10^{-5}$, System~\eqref{eq:redmodel} has three equilibria: the $DFE$ stable node, the $EE_1$ unstable node, and the $EE_2$ stable node ($\beta_c =1.25\times 10^{-5}$ is the critical value of $\beta$, where the backward starts, and this segment of values is part of the backward bifurcation region defined in next sections),
	\item for $1.29\times 10^{-5} \leq \beta \leq 1.33 \times 10^{-5}$, System~\eqref{eq:redmodel} has three equilibria: the $DFE$ stable node, the $EE_1$ unstable node, and the $EE_2$ stable focus (is part of the backward bifurcation region),
	\item for $1.33\times 10^{-5} \leq \beta \leq 1.89 \times 10^{-5}$, System~\eqref{eq:redmodel} has two equilibria: the $DFE$ unstable node, and the $EE_2$ stable focus,
	\item for $1.89\times 10^{-5} \leq \beta \leq 2.33 \times 10^{-5}$, System~\eqref{eq:redmodel} has two equilibria: the $DFE$ unstable node, and the $EE_2$ unstable focus, which is surrounded by a limit cycle ($\beta_c^{(1)}=1.89\times 10^{-5}$ is the first critical value of $\beta$ where the Hopf bifurcation starts, and the segment is the Hopf bifurcation region, or the \textit{bubble}, defined in the next sections),
	\item for $2.33 \times 10^{-5} \leq \beta \leq 2.5 \times 10^{-5}$, System~\eqref{eq:redmodel} has two equilibria: the $DFE$ unstable node, and the $EE_2$ stable focus, as before the Hopf bifurcation ($\beta_c^{(2)}=2.33 \times 10^{-5}$ is the second critical value of $\beta$ where the Hopf bifurcation ends). 
\end{enumerate}
\begin{figure}[H]
	\centering
	\includegraphics[width=0.75\textwidth]{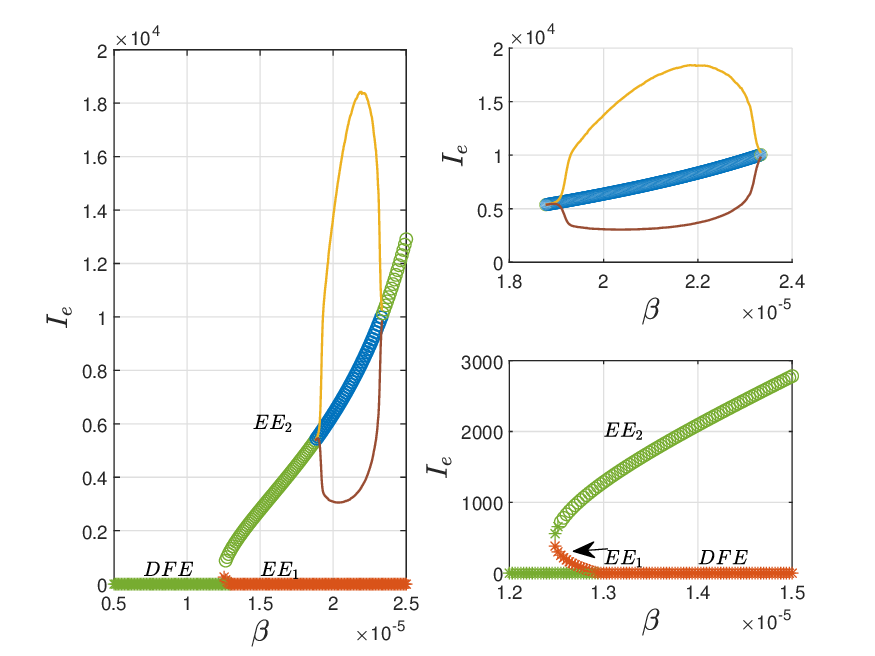}
	\caption{Left panel: prevalence bifurcation over $\beta$. Green stars: stable nodes, red stars: unstable nodes, green circles: stable focuses, blue circles: unstable focuses surrounded by a stable limit cycle. Right lower panel: backward bifurcation region. Right upper panel: Hopf bifurcation region. The yellow and violet curves show the amplitude of the limit cycle. }
	\label{fig:bifu1_beta}
\end{figure}
As it can be seen, there are two main abnormal behaviors, a backward and a Hopf bifurcation. The first one is a subcritical pitchfork bifurcation, and consists in two stable equilibria with an unstable one in between. States above the unstable $EE_1$ are attracted by the stable $EE_2$ (whether this is a stable node or a stable focus); states below the unstable equilibrium $EE_1$ are attracted by the stable $DFE$. The second bifurcation is a supercritical Hopf bifurcation, and consists in the emergence (and as $\beta$ increases, the annihilation) of unstable focuses for $EE_2$, and stable limit cycles `around' them (the states that escape from each unstable focus converge to a particular limit cycle, whose amplitude vary continuously as $\beta$ increases). Both, backward and Hopf bifurcations will be analyzed in detail in the next sections.

Figure~\ref{fig:bifu1_phi}, on the other hand, shows the $I$-component of the equilibria when $1.5 \leq  \phi \leq 4$ and $\beta$ is fixed at its nominal value $\beta^{nom}$. The equilibria behavior is qualitatively the reverse of that of Figure~\ref{fig:bifu1_beta}, and also shows six different region (as $\phi$ decreases).
\begin{enumerate}
	\item For $4 \geq \phi \geq 3.02$, System~\eqref{eq:redmodel} has only the $DFE$ stable node,
	\item for $ 3.02 \geq \phi \geq 2.98$, System~\eqref{eq:redmodel} has three equilibria: the $DFE$ stable node, the $EE_1$ unstable node, and the $EE_2$ stable node ($\phi_c=3.02$ is the critical value of $\phi$, where the backward starts, and the segment is part of the backward bifurcation region),
	\item for $2.98 \geq \phi \geq 2.86$, System~\eqref{eq:redmodel} has three equilibria: the $DFE$ stable node, the $EE_1$ unstable node, and the $EE_2$ stable focus (is part of the backward bifurcation region),
	\item for $2.86 \geq \phi \geq 1.96$, System~\eqref{eq:redmodel} has two equilibria: the $DFE$ unstable node, and the $EE_2$ stable focus,
	\item for $1.96 \geq \phi \geq 1.8$, System~\eqref{eq:redmodel} has two equilibria: the $DFE$ unstable node, and the $EE_2$ unstable focus, which is surrounded by a limit cycle ($\phi_c^{(2)}=1.96$ is the second critical value of $\phi$, where the Hopf bifurcation starts, and the segment is the Hopf bifurcation region or \textit{bubble}),
	\item for $1.8 \geq \phi \geq 1.5$, System~\eqref{eq:redmodel} has two equilibria: the $DFE$ unstable node, and the $EE_2$ stable focus, as before the Hopf bifurcation ($\phi_c^{(1)}=1.8$ is the first critical value of $\phi$, where the Hopf bifurcation ends).
\end{enumerate}
\begin{figure}[H]
	\centering
	\includegraphics[width=0.75\columnwidth]{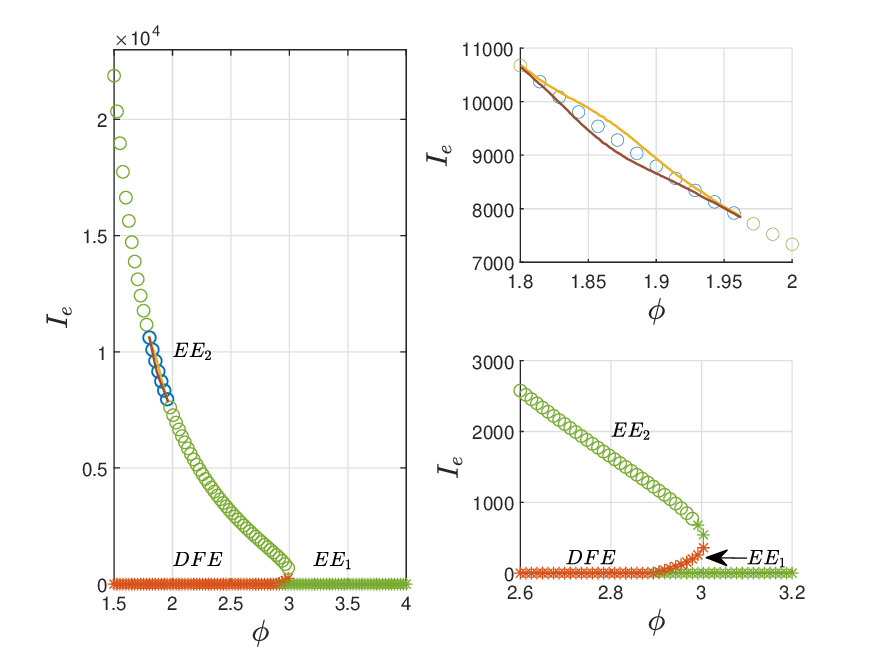}
	\caption{Left panel: prevalence bifurcation over $\phi$. Green stars: stable nodes, red stars: unstable nodes, green circles: stable focuses, blue circles: unstable focuses stable and stable limit cycles. Right lower panel: backward bifurcation region. Right upper panel: Hopf bifurcation region. The yellow and violet curves show the amplitude of the limit cycle. }
	\label{fig:bifu1_phi}
\end{figure}

\subsection{Backward bifurcation (subcritical pitchfork bifurcation)}

A backward bifurcation can be viewed as a particular case of a pitchfork bifurcation, in which a system transitions ---when a parameter crosses a critical value--- from one to three equilibria. As with Hopf bifurcations, pitchfork bifurcations can be supercritical or subcritical.

In a supercritical pitchfork bifurcation, a stable equilibrium becomes unstable and gives rise to two new stable equilibria. In epidemiological models, only non-negative equilibria are relevant, so this case corresponds to the classical forward bifurcation: when the bifurcation parameter crosses its critical value, the disease-free equilibrium loses stability and a stable endemic equilibrium emerges \cite{strogatz2018nonlinear}.

In contrast, a subcritical pitchfork bifurcation involves a single unstable equilibrium at zero becoming stable, while two unstable equilibria (one positive and one negative) emerge. A particularly important case arises when, before the critical value at which the zero equilibrium changes stability, two unstable nonzero equilibria and two additional stable nonzero equilibria appear. In epidemiology (i.e., considering only positive equilibria), this scenario is known as a backward bifurcation and significantly alters the system’s behavior. The unstable endemic equilibrium separates the basins of attraction of the DFE and the stable endemic equilibrium.

A key consequence of backward bifurcation is hysteresis \cite{strogatz2018nonlinear}: the system’s behavior depends on the direction of parameter variation. To illustrate this, consider the right lower panel of Figure~\ref{fig:bifu1_beta} (reproduced in Figure~\ref{fig:hysteresis}). Suppose the system starts at the stable DFE corresponding to $\beta < \beta_c=1.25 \times 10^{-5}$. As $\beta$ increases slowly, the system remains near the DFE, adjusting continuously\footnote{Recall that dynamical system cannot moved along an equilibrium manifold without abandon them.}, until $\beta$ reaches $1.33 \times 10^{-5}$, at which point the DFE loses its stability and the system abruptly transitions to the stable endemic equilibrium $EE_2$.

Conversely, suppose the system is initially at the stable endemic equilibrium $EE_2$ for $\beta > 1.33 \times 10^{-5}$, with a relatively high value of $I_e$ (e.g., as observed in the case study of Salta, Argentina). If $\beta$ is then decreased slowly, the system follows the branch of stable endemic equilibria until $\beta$ reaches $1.25 \times 10^{-5}$, where it abruptly transitions back to the DFE, the only stable equilibrium below this threshold.

In summary, the epidemic emerges at $\beta = 1.33 \times 10^{-5}$ (corresponding to $R_0 = 1$), but disappears only at $\beta = 1.25 \times 10^{-5}$ (where $R_0 < 1$). This asymmetry, and the abrupt changes in the prevalence, are a hallmark of backward bifurcations. Figure~\ref{fig:hysteresis} illustrates both the evolution of equilibria as $\beta$ varies (solid arrows) and the sudden jumps associated with the onset and disappearance of the epidemic (dotted arrows).

An analogous backward bifurcation behavior is observed for the parameter $\phi$, but with the roles of increasing and decreasing parameter values reversed.
\begin{figure}[H]
	\centering
	\includegraphics[width=.75\columnwidth]{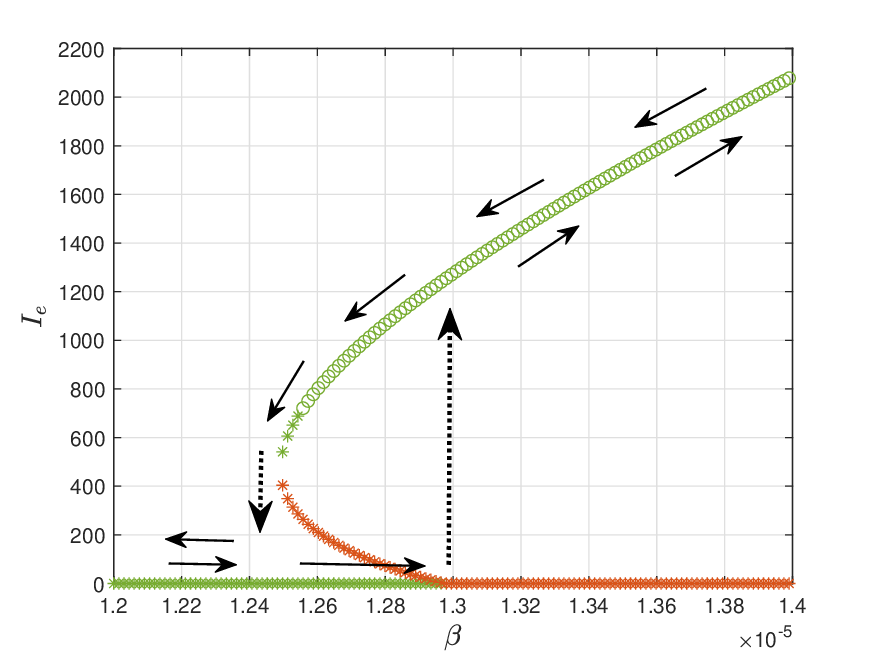}
	\caption{Hysteresis phenomenon produced by the backward bifurcation in $I_e$, when $\beta$ changes. The black arrows denotes, symbolically, the system evolution as $\beta$ increases and decreases. The black dotted arrows denotes the jumps of $I_e$ when the epidemic emerges (with $\beta$ increasing), and disappear (with $\beta$ decreasing).}
	\label{fig:hysteresis}
\end{figure}
\subsubsection{Backward bifurcation in {$\Omega$}}

We already see that for the nominal values of $\phi$ and $\beta$, the system exhibits a backward bifurcation in both, $\beta$ and $\phi$, respectively (Figures~\ref{fig:bifu1_beta}~and~\ref{fig:bifu1_phi}). We now extend this analysis to characterize all combinations of $\beta$ and $\phi$ within the parameter domain $\Omega$.

The formal criterion for a prevalence backward bifurcation is the coexistence of $DFE$, $EE_1$, and $EE_2$, with $I_e^{(1)}<I_e^{(2)}<I_e^{(3)}$, being $EE_1$ unstable and $DFE$ and $EE_2$, stable. According to the analysis made in Sections~\ref{sec:existence}~and~\ref{sec:stabilityeq}, the region in which these three equilibria coexist and fulfill the previous inequality is given by the subset $\Omega_3$ (the red region in Figure~\ref{fig:twoplots_b}). Figure~\ref{fig:bifurcations2D} illustrates the resulting BBR in $\Omega$ (yellow region).

\begin{figure}[H]
	\centering
	\includegraphics[width=.65\columnwidth]{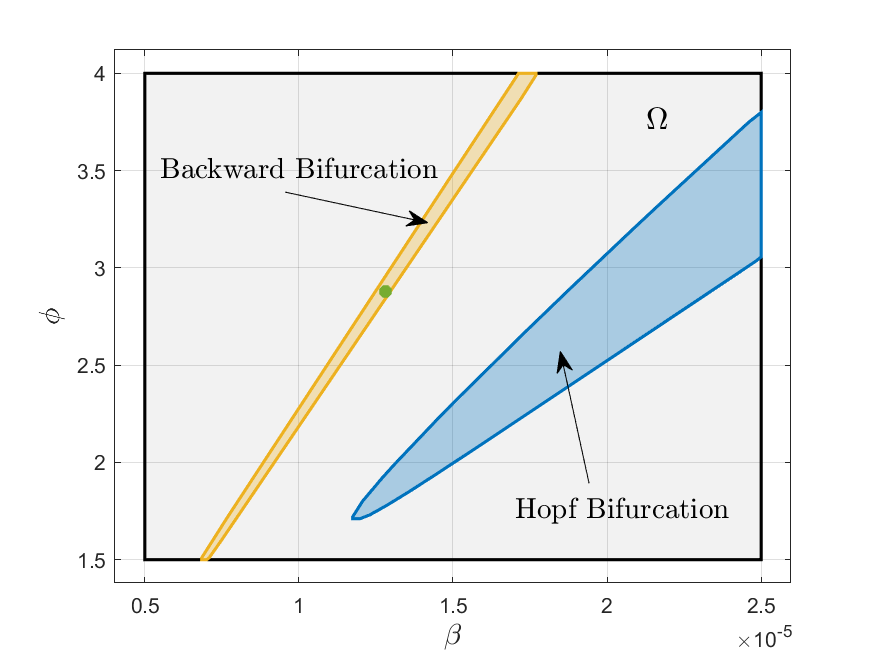}
	\caption{Backward and Hopf bifurcation regions in $\Omega$. The yellow region represents the BBR while the blue one represents the HBR. The green dot represents the nominal values of $\beta$ and $\phi$, $\beta^{nom}$ and $\phi^{nom}$.}
	\label{fig:bifurcations2D}
\end{figure}

\subsection{Supercritical Hopf bifurcation} \label{sec:hopfbif}

A Hopf (or Poincar\'e-Andronov-Hopf) bifurcation occurs when varying one or more parameters of a system causes its solutions to change from being attracted to (or repelled from) an equilibrium point to being attracted to (or repelled from) a periodic (oscillatory) solution. In particular, a supercritical Hopf bifurcation occurs when a stable focus becomes unstable at a critical parameter value, as the parameter increases or decreases. In many cases, the resulting motion is a small-amplitude limit cycle around the unstable equilibrium, whose amplitude grows as the parameter moves away from the critical value into the instability region. In many cases, also, there exists a second critical value of the parameter where the equilibrium return to be a stable focus. This phenomenon; i.e., the existence of a finite region of parameters values associated with unstable foucs and stable limit cycles, is known as bubble.

To illustrate this phenomenon, consider the upper-right panel of Figure~\ref{fig:bifu1_beta}, where only the $\beta$ bifurcation is depicted. Suppose that the system operates in the stable focus region of $EE_2$, i.e., with $\beta < \beta_c^{(1)} = 1.89 \times 10^{-5}$. In this case, any trajectory starting away from $EE_2$ converges to it via damped oscillations, since $EE_2$ is the only attracting steady state. However, as $\beta$ increases, the pair of complex conjugate eigenvalues of the Jacobian matrix evaluated at $EE_2$ crosses the imaginary axis from left to right (the third eigenvalue remains real and negative) at $\beta = \beta_c^{(1)}$. The focus then becomes unstable. Even if the system initially is at $EE_2$, any perturbation (including the variation of $\beta$ itself) causes the trajectory to move away from it, following oscillations of increasing amplitude and eventually converging to a limit cycle (whose amplitude and frequency depend on $\beta$). In fact, any initial condition leads the system to converge to a limit cycle, either from the inside or the outside, since stability has shifted from $EE_2$ to the periodic orbits.

To better show this effect in the time domain, we perform two simulations with initial conditions inside and outside the limit cycle, with $\beta = 2.1054 \times 10^{-5}$, which is a value between $\beta_c^{(1)}$ and $\beta_c^{(2)}$. The evolution of $I(t)$ and the phase portrait of the system (in the space $U$, $I$, $E$) are shown in Figures~\ref{fig:limicycle1}~and~\ref{fig:limicycle2}, respectively. Observe that the amplitude of the limit cycle (and the oscillations in $I(t)$) could be significantly large, causing the prevalence to reach unacceptable high values.

\begin{figure}[H]
	\centering
	\begin{minipage}{0.45\columnwidth}
		\centering
		\includegraphics[width=1.1\linewidth]{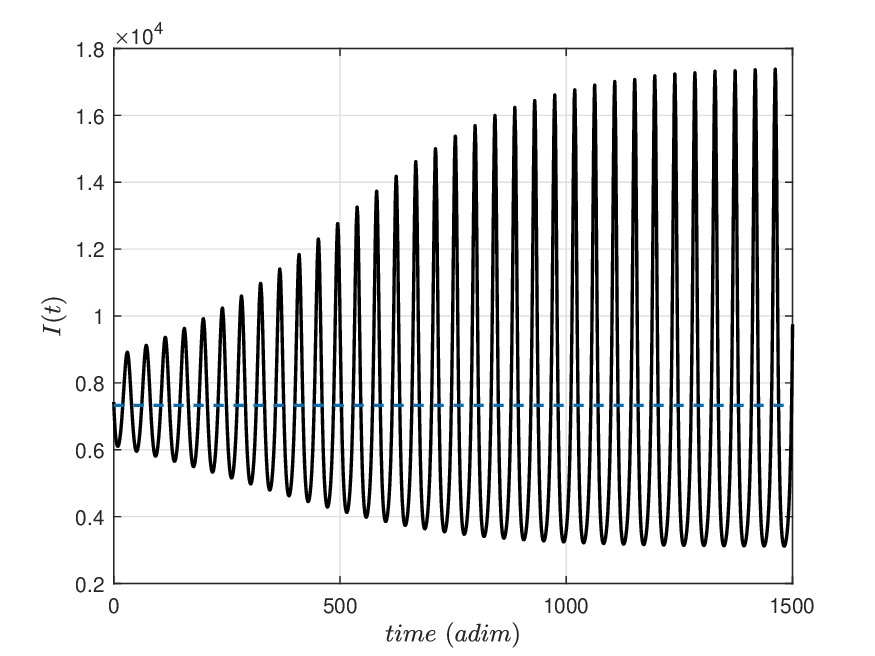}
	\end{minipage}
	\hfill
	\begin{minipage}{0.45\columnwidth}
		\centering
		\includegraphics[width=1.1\linewidth]{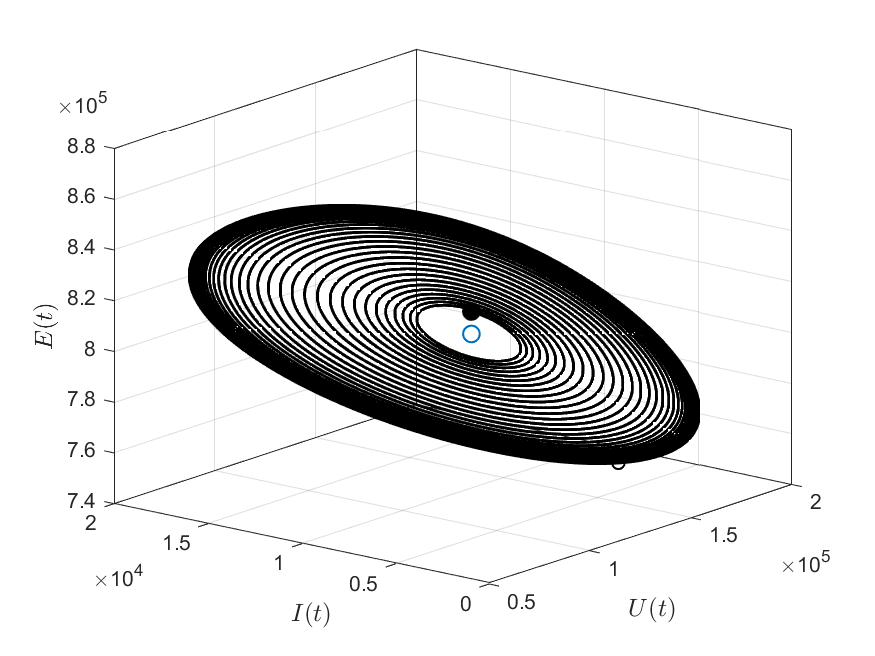}
	\end{minipage}
	\caption{Left: evolution of the prevalence $I(t)$ for $\beta=2.1054\times10^{-5}$ (black solid line) and the equilibrium value (blue dashed line). Right: phase portrait for the initial condition $x_{0}=[1.1020\times10^5;\ 7424;\ 8.1993\times10^5]$ (black filled circle), inside the limit cycle. The blue empty circle represents the equilibrium $EE_2$.}
	\label{fig:limicycle1}
\end{figure}

\begin{figure}[H]
	\centering
	\begin{minipage}{0.45\columnwidth}
		\centering
		\includegraphics[width=1.1\linewidth]{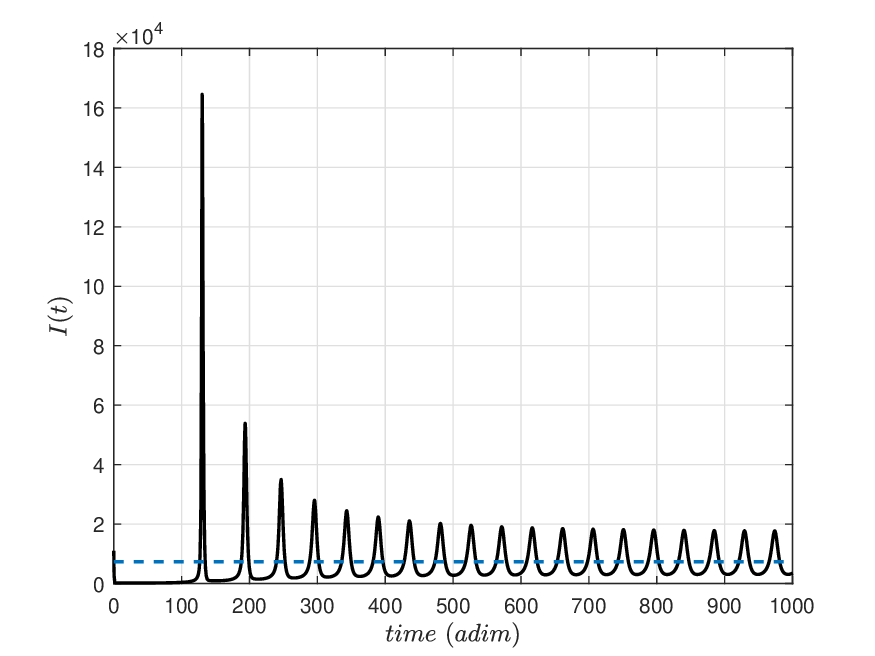}
	\end{minipage}
	\hfill
	\begin{minipage}{0.45\columnwidth}
		\centering
		\includegraphics[width=1.1\linewidth]{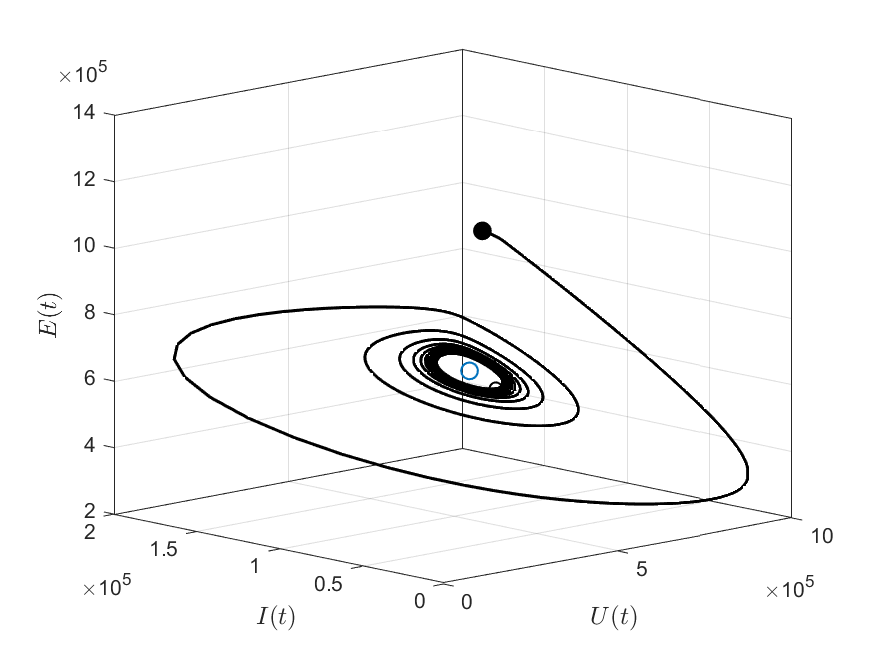}
	\end{minipage}
	\caption{Left: evolution of the prevalence $I(t)$ for $\beta=2.1054\times10^{-5}$ (black solid line) and the equilibrium value (blue dashed line). Right: phase portrait for the initial condition $x_{0}=[1.6365\times10^5;\ 11012;\ 1.2177\times10^6]$ (black filled circle), outside the limit cycle. The blue empty circle represents the equilibrium $EE_2$.}
	\label{fig:limicycle2}
\end{figure}

An analogous Hopf bifurcation (and bubble) is observed with respect to the parameter $\phi$, but with the roles of increasing and decreasing parameter values reversed (see Figure~\ref{fig:bifu1_phi}).

\subsubsection{Hopf bifurcation in $\Omega$}

We have already seen that, for the nominal values of $\phi$ and $\beta$, the system undergoes a Hopf bifurcation with respect to both parameters (see Figures~\ref{fig:bifu1_beta}~and~\ref{fig:bifu1_phi}). We now extend this analysis to characterize all combinations of $\beta$ and $\phi$ within the parameter domain $\Omega$.

To determine the Hopf bifurcation region (HBR) in $\Omega$, we analyze the Jacobian~\eqref{eq:jacob3D} evaluated at $EE_2$~\cite{rionero2019hopf}. The characteristic polynomial of the Jacobian is given by
\begin{eqnarray} \label{Hopf}
	\lambda^{3} - z_{1}\lambda^{2} + z_{2}\lambda - z_{3} = 0,
\end{eqnarray}
where
\begin{align*}
	z_{1} &= J_{11} + J_{22} + J_{33}, \\
	z_{2} &=
	\det\begin{pmatrix}
		J_{11} & J_{12} \\ 
		J_{21} & J_{22}
	\end{pmatrix}
	+
	\det\begin{pmatrix}
		J_{11} & J_{13} \\ 
		J_{31} & J_{33}
	\end{pmatrix}
	+
	\det\begin{pmatrix}
		J_{22} & J_{23} \\ 
		J_{32} & J_{33}
	\end{pmatrix}, \\
	z_{3} &= \det(J),
\end{align*}
with $J = J(EE_2, \beta, \phi, \tilde p_{r})$ being the Jacobian matrix, and $J_{11}$, $J_{12}$, $J_{21}$, and $J_{22}$ the elements of $J$. Note that the coefficients depend on the system parameters and on $EE_2$, which also depends on the parameters.

The conditions for $EE_2$ to be a focus are
\begin{eqnarray}
	z_{1}(\beta,\phi,\tilde p_{r}) &<& 0, \nonumber \\
	z_{2}(\beta,\phi,\tilde p_{r}) &>& 0, \nonumber \\
	z_{3}(\beta,\phi,\tilde p_{r}) &<& 0,
\end{eqnarray}
while the focus is stable when $z_{3} > z_{1} z_{2}$ and unstable when $z_{3} < z_{1} z_{2}$. The boundary between both regimes occurs when
\begin{eqnarray}\label{hopf condition}
	z_{1}(\beta,\phi,\tilde p_{r}) &<& 0, \nonumber \\
	z_{2}(\beta,\phi,\tilde p_{r}) &>& 0, \nonumber \\
	z_{3}(\beta,\phi,\tilde p_{r}) &=& z_{1}(\beta,\phi,\tilde p_{r})\, z_{2}(\beta,\phi,\tilde p_{r}),
\end{eqnarray}
in which case the characteristic polynomial~\eqref{Hopf} has a pair of purely imaginary roots $\pm i\sqrt{z_2}$ and one negative real root $z_1$. So, for a fixed $\tilde p_r$, condition~\eqref{hopf condition} defines the boundary of the HBR in $\Omega$. After some algebraic manipulation of the equation $z_3=z_1 z_2$ (see Appendix), we obtain the following implicit function of $\beta$ and $\phi$:
\begin{align}\label{eq:HBR}
c_1\beta^3+c_2\beta^2\phi+c_3\beta^2+c_4\beta\phi^2+c_5\beta\phi+c_6\beta+c_7\phi^2+c_8\phi+c_9=0,
\end{align}
where the coefficients $c_i(\beta,\phi,\tilde p_{r})$, $i = 1, \ldots, 9$, are given in the Appendix. The HBR is depicted in Figure~\ref{fig:bifurcations2D} (blue region), and the frequency of the limit cycles near the boundary can be approximated by $\sqrt{z_2}$.

To establish that condition~\eqref{hopf condition} yields a genuine, supercritical Hopf bifurcation ---rather than a mere sign change of $\mathrm{Re}\,\lambda$--- the two standard genericity conditions must be verified.

\begin{proposition}
Let $\tilde p_r$ be fixed and let $\gamma\subset\Omega$ be a smooth path crossing the curve~\eqref{hopf condition} at a point $(\beta_c,\phi_c)$. Then System~\eqref{eq:redmodel} undergoes a supercritical Hopf bifurcation of $EE_2$ at $(\beta_c,\phi_c)$, giving rise to a stable limit cycle.
\end{proposition}

\begin{proof} At a point satisfying~\eqref{hopf condition}, $J(EE_2)$ has a pair of purely imaginary eigenvalues $\pm i\sqrt{z_2}$ and a negative real eigenvalue $z_1$. Two genericity conditions remain. \emph{(i) Transversality.} By the eigenvalue-free criterion of~\cite{liu1994criterion}, the complex pair crosses the imaginary axis with nonzero speed if and only if $\frac{d}{d\mu}\!\left(z_3-z_1z_2\right)\neq 0$ along the path, where $\mu$ denotes the parameter being varied ($\beta$ or $\phi$); this holds throughout the boundary~\eqref{eq:HBR}. \emph{(ii) Non-degeneracy.} The first Lyapunov coefficient $\ell_1$, computed through the projection method of~\cite{kuznetsov2004elements}, is negative at the Hopf points bounding the bubbles in $\beta$ ($\beta_c\approx1.88\times10^{-5}$ and $2.33\times10^{-5}$) and in $\phi$ ($\phi_c\approx1.80$ and $1.96$). Hence the bifurcation is supercritical and the emerging limit cycle is stable. 
\end{proof}

\begin{rem}
    The existence of attractive limit cycles sheds light on some behaviors experienced in reality, where the epidemic prevalence oscillates over time. Public health authorities tend to look for an external cause to explain the oscillatory behavior, while the intrinsic dynamics of the system is indeed the only cause.
\end{rem}

\subsection{Structural emergence of backward and Hopf bifurcations}

Next, we analyze the structural conditions for the emergence of backward and Hopf bifurcations. The parameters $\beta$ and $\phi$ can be externally adjusted to influence the epidemic dynamics (they can be interpreted as manipulated variables in a control framework), while the remaining parameters are assumed to be constant. Nevertheless, these fixed parameters are typically obtained through identification procedures and may vary under slightly different conditions. Therefore, a robustness analysis of the existence of bifurcations can provide further insight into the overall system behavior.

Among the parameters in $p_r$, there exists a key parameter that drives the emergence of backward and Hopf bifurcations when $\beta$ and/or $\phi$ are varied. In particular, the reinfection rate coefficient associated with the latent class, $k$, plays a crucial role in shaping the system dynamics, as it introduces an additional pathway to the infectious class $I$ through contagion. Figure~\ref{fig:Sen_k} presents the $\beta$-bifurcation diagram (with $\phi$ fixed at its nominal value) and the $\phi$-bifurcation diagram (with $\beta$ fixed at its nominal value) for different values of $k$. 

As shown, there exists a range of values of $k$, including the nominal value obtained from the identification process ($k=0.26$), for which the system exhibits nonstandard behavior, namely the coexistence of three equilibria and the emergence of stable limit cycles. As $k$ increases from $0.15$ to $0.2$, the system transitions from a forward to a backward bifurcation, which is consistent with classical results in \cite{feng2000model,castillo2004dynamical}. More interestingly, for $k \geq 0.25$, the system not only exhibits backward bifurcations but also supercritical Hopf bifurcations (also reported in some works \cite{bonyah2020hopf,das2025exploring}, but in the context of delayed treatments). 

For large values of $\beta$, or equivalently small values of $\phi$, the endemic equilibrium may lose stability (transitioning from a stable focus to an unstable focus and then back to stable), leading to the emergence of stable limit cycles with large amplitude. This phenomenon—namely, the appearance and subsequent disappearance of limit cycles—is referred to as a \textit{bubble} in recent works (see, e.g., \cite{das2024exploring}), and constitutes an interesting behavior from both modeling and control design perspectives.

According to \cite{wangari2018backward}, the parameter $k$ takes a value of $k=0.26$ (our nominal value), which lies within the range that leads to the emergence of both backward and Hopf bifurcations and motivates the previous bifurcation analysis.
\begin{figure}[H]
\centering
\begin{minipage}{0.4\columnwidth}
    \centering
    \includegraphics[width=1.1\linewidth]{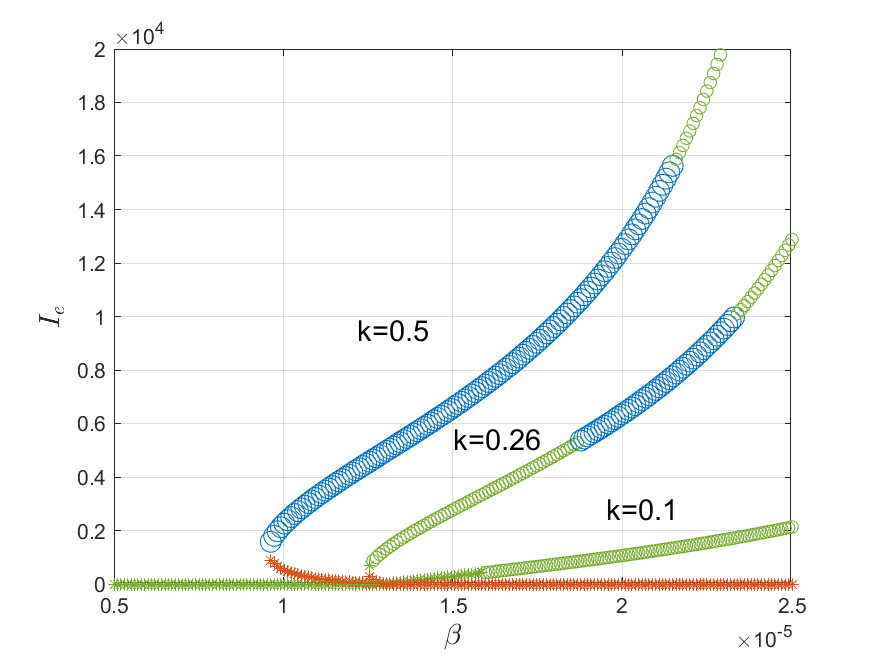}
\end{minipage}
\hfill
\begin{minipage}{0.4\columnwidth}
    \centering
    \includegraphics[width=1.1\linewidth]{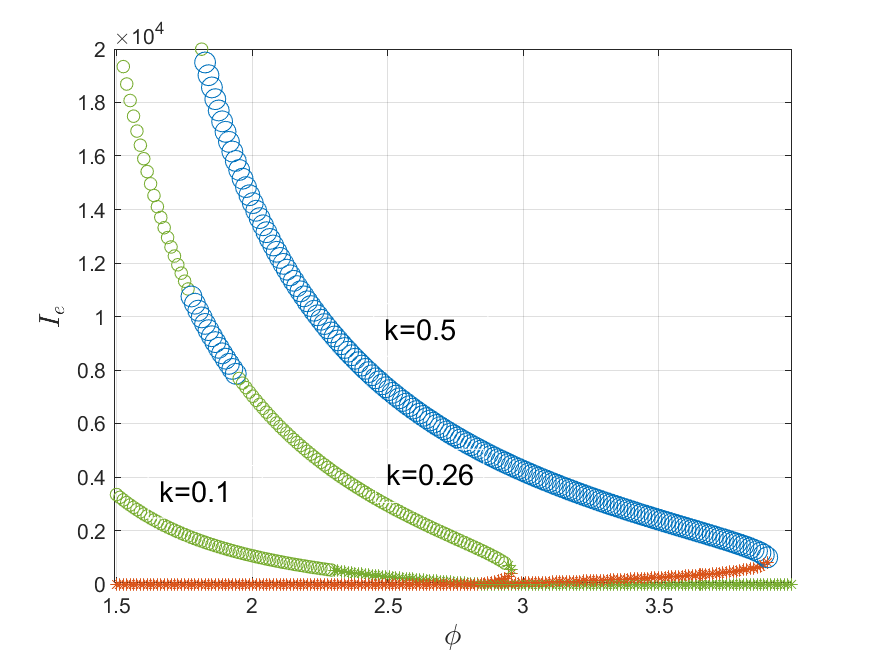}
\end{minipage}
        \caption{$\beta$-Bifurcation (with $\phi$ taking the nominal value) and $\phi$-Bifurcation (with $\beta$ taking the nominal value) diagrams for System~\eqref{eq:model} considering different values of $k$. The nominal value is $k=0.26$. The curves shows the existence of equilibria and their stability.}
        \label{fig:Sen_k}
\end{figure}

\section{Interventions strategies: some preliminary intuitions}\label{sec:control}

Building on the bifurcation structure characterized in the previous section, we now discuss preliminary control strategies that exploit the geometry of the parameter space $\Omega$. Figure~\ref{fig:bifu_Omega} presents the prevalence bifurcations, along with a zoomed view of the region around the backward bifurcation, for clarity. Bifurcations on the other state's equilibrium components ($U_e$ and $E_e$) are equivalents, since they are explicit functions of $I_e$.
\begin{figure}[H]
	\centering
	\includegraphics[width=1.1\columnwidth]{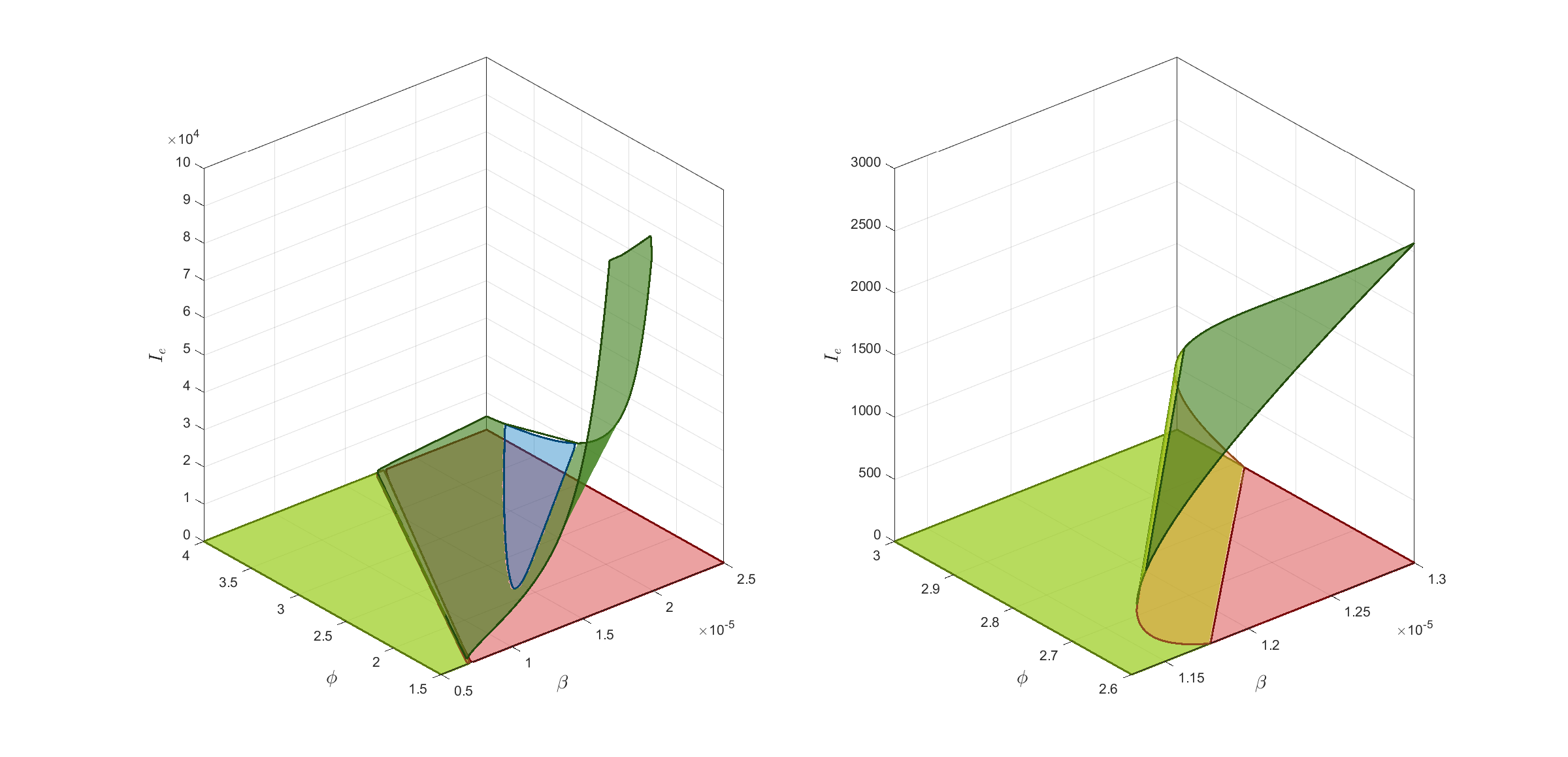}
	\caption{Left: prevalence bifurcation in $\Omega$. The light-green region corresponds to stable nodes, the yellow region represents the $BBR$, the dark-green region corresponds to a stable focus, and the blue region represents the $HBR$. Right: zoomed view in the vicinity of the $BBR$. Figures~\ref{fig:bifu1_beta}~and~\ref{fig:bifu1_phi} correspond to slices of this complete plot.}
	\label{fig:bifu_Omega}
\end{figure}

The general dynamical and bifurcation analysis suggest that finding the proper combination of 'case finding' and 'case holding' control actions (i.e., modifications of $\beta$ and $\phi$) to reduce the prevalence to undetectable levels is far from trivial. In this first approach, we do not go into the details of an optimization-based control strategy, but of intuitive strategies exploiting the geometry of bifurcations. 

As in many studies, we can plot $R_0$ as a function of $\beta$ and $\phi$, as shown in Figure~\ref{fig:R0_beta_phi}, with the objective of simply reducing it. However, this plot does not indicate the best way to reduce $I$ through a proper combination of decreasing $\beta$ and increasing $\phi$, since bringing $R_0$ below $1$ does not necessarily imply the end of the epidemic and the Hopf bifurcation is not transparently related to $R_0$.
\begin{figure}[H]
    \centering
	\includegraphics[width=.75\columnwidth]{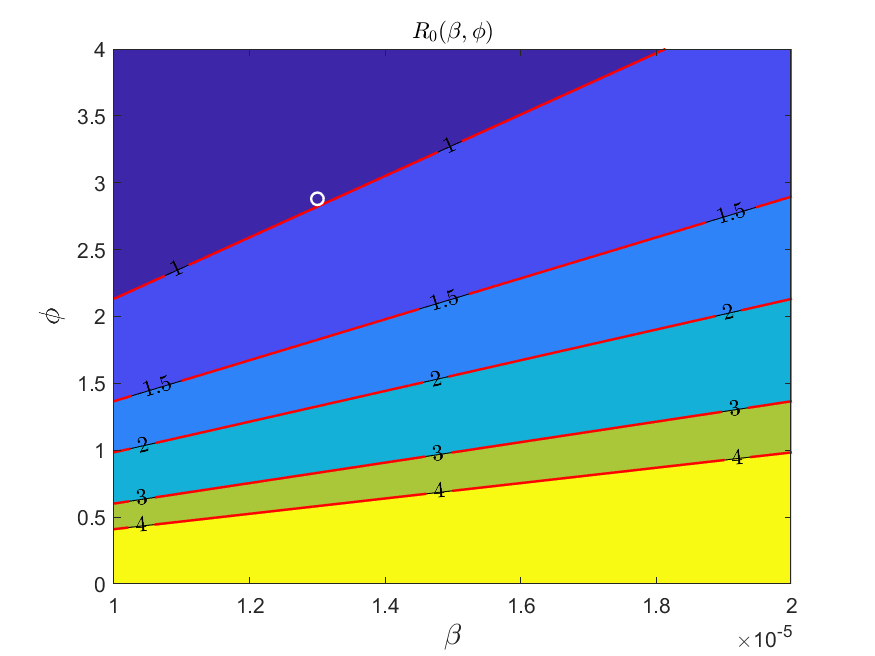}
        \caption{$R_0$ as a function of $\beta$ and $\phi$, for System~\eqref{eq:model} The white dot represents $R_0$ the nominal values of $\beta$ and $\phi$ ($R_0=0.9647$).}
        \label{fig:R0_beta_phi}
\end{figure}

A more informative assessment can be obtained by analyzing Figures~\ref{fig:bifurcations2D}~and~\ref{fig:bifu_Omega}, and design control interventions to avoid undesirable behaviors. 

At first glance, it becomes clear that allocating effort exclusively to increasing $\phi$ without reducing $\beta$ (see the evolution of $I_e$ for $\beta = 2\times 10^{-5}$), or vice versa (see the evolution of $I_e$ for $\phi = 2$), is not optimal for avoiding limit cycles (and the resulting oscillations and increase in $I$). This implies that focusing solely on either case finding or case holding is insufficient to steer the system toward a low $I_e$ without first experiencing a significant rise.

A more detailed inspection of the 3D plot suggests that one can design strategies to "avoid" the limit cycle region and approach the line defining the drop of $I_e$ below detectable levels (given by the upper boundary of the BBR, shown in yellow in Figure~\ref{fig:bifurcations2D}) perpendicularly. For example, when starting with a high $\beta$ and a middle/high $\phi$ (e.g., $\beta=2.5\times10^{-5}$ and $\phi=2.5$), it is preferable to initially reduce both $\phi$ and $\beta$ (even though reducing $\phi$ may seem counterintuitive), thereby skirting the HBR (blue region in Figure~\ref{fig:bifurcations2D}). Subsequently, one should increase $\phi$ and further decrease $\beta$ to approach the drop line of $I_e$ as quickly as possible.

Conversely, when starting with low $\beta$ and low $\phi$ (e.g., $\beta=1\times10^{-5}$ and $\phi=1.5$), the best strategy is to move perpendicularly toward the drop line of $I_e$. Starting from high $\beta$ and low $\phi$ (the worst-case scenario) can be viewed as a particular instance of the first case. The 3D bifurcation plot suggests that the most effective intervention is to first reduce $\beta$ and then increase $\phi$ to circumvent the HBR, before proceeding directly toward the drop line of $I_e$.

If the system is already in a limit cycle, the best policy is to quickly steer it out of the blue region; that is, to find a combination of $\beta$ and $\phi$ that moves the system perpendicularly to the upper boundary of the HBR.

Overall, the analysis suggests that both timing and combination of interventions are crucial. Therefore, an optimization-based control strategy should explicitly incorporate the system’s bifurcation structure: simply minimizing $I$ (or even $I+L$) subject to box-type constraints on case finding and case holding actions is not sufficient.

\begin{figure}[H]
	\hspace*{-1.6cm}\includegraphics[width=1.2\textwidth]{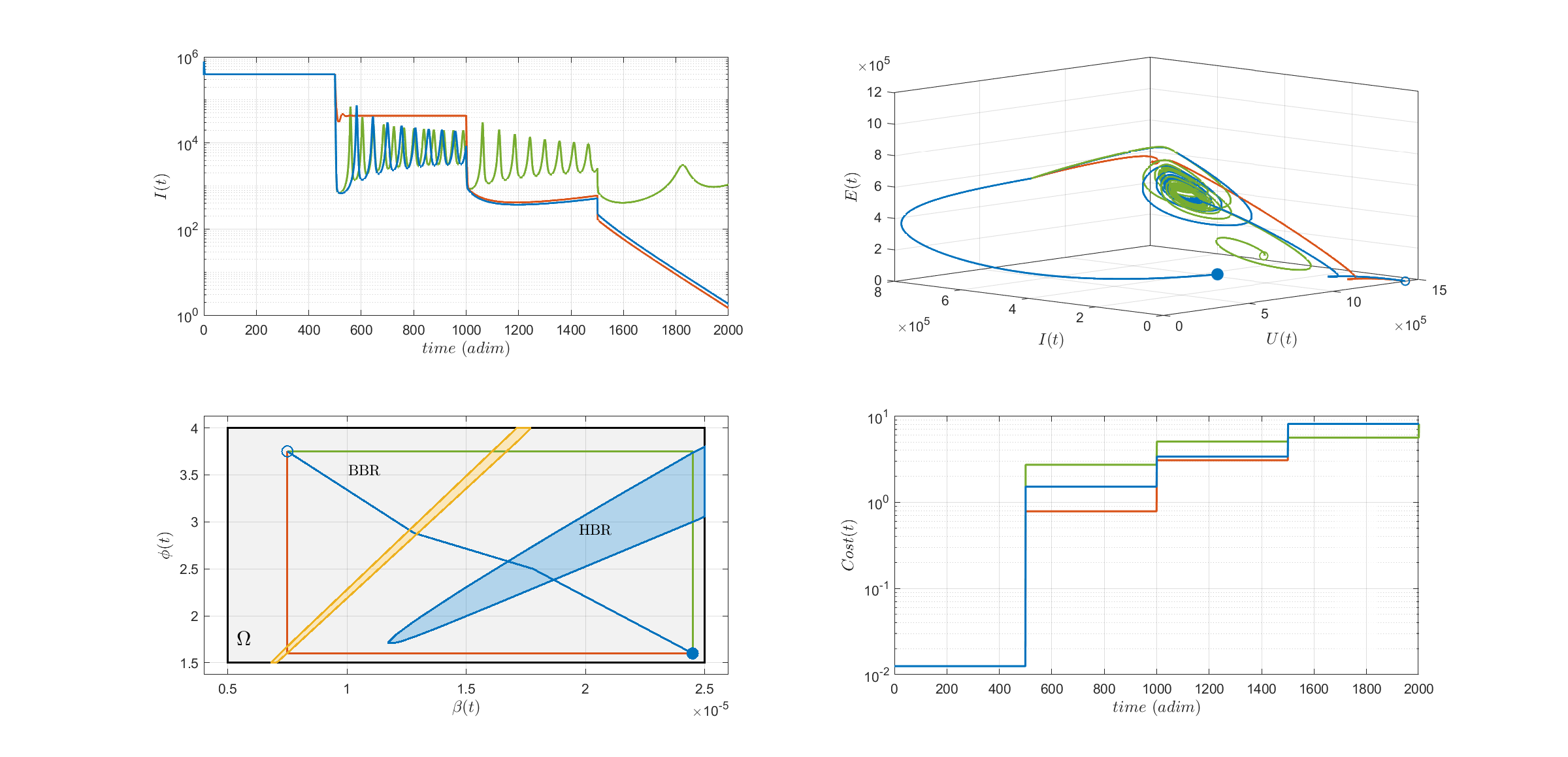}
	\caption{Three simple control strategies to reach undetectable prevalence value were tested. Control 1 (green lines) increases first $\phi$ and then decrease $\beta$, control 2 (blue lines) decreases $\beta$ and increases $\phi$ simultaneously, and control 3 (red line) decrease first $\beta$ and then increases $\phi$. Left upper panel: time evolution of infected $I(t)$. Right upper panel: phase portrait. Left lower panel: control path in $\Omega$. Right lower panel: time evolution of the cost function $Cost(t)$.}
	\label{fig:controlTB2}
\end{figure}

These preliminary results suggest that an optimization-based control strategy should explicitly incorporate the bifurcation structure of the system, rather than relying solely on $R_0$-based criteria.

\section{Conclusions}

In this work, a model for tuberculosis epidemics was introduced and characterized. In agreement with the specialized literature, it includes reinfections from treated and latent individuals, leading to the emergence of subcritical pitchfork (backward) and supercritical Hopf (limit cycle emergence) bifurcations. Notably, the Hopf bifurcation is shown to emerge without delays or saturation effects, which distinguishes this work from much of the existing literature. A multivariate bifurcation analysis ---rather than a classical $R_0$-based approach--- was conducted by focusing on parameters directly associated with case finding and case holding interventions, which are the primary tools for managing this type of epidemic at the healthcare level.
Data from Salta, a province of Argentina, were selected as a case study; the qualitative results extend naturally to other high-burden settings beyond Salta.

Based on an exhaustive dynamical analysis, preliminary control strategies were envisioned to avoid undesirable outcomes in healthcare interventions. More precisely, the model characterization and, particularly, the bifurcation analysis, suggest that such dynamical information should be explicitly incorporated into optimization-based control strategies to prevent counterproductive - or even detrimental - effects. Future works include developing a formal procedure to find optimal policies (i.e., specifying the time-dependent allocation of efforts between $\beta$ and $\phi$) through a formal optimal control problem. Additionally, the model structure could be refined ---potentially incorporating additional compartments and reinfection pathways--- and parameter estimation methods could be further improved.

\section{Bibliography}


\begin{thebibliography}{99}
\setlength{\itemsep}{0pt}\setlength{\parsep}{0pt}
\bibitem{BoletinN7TBArg}
N.~Alvarez, L.~Angueira, P.~Barletta, M.~Chernomoretz, S.~Ioannoy,
  J.~Linares~Pessacq, C.~Roncoroni, S.~Arias, G.~Armando, J.~C. Bossio,
  R.~Colombini, L.~Corti, R.~Dalla~Fontana, Y.~Díaz, H.~Fernández, L.~Ghio,
  R.~López, A.~Mancinelli, N.~Mordini, and I.~Vanni, ``Boletín sobre
  tuberculosis y lepra en la {A}rgentina,'' no.~7, pp. 1--102, 2024.

\bibitem{blower1996control}
S.~Blower, P.~Small, and P.~Hopewell, ``Control strategies for tuberculosis
  epidemics: new models for old problems,'' \emph{Science}, vol. 273, no. 5274,
  pp. 497--500, 1996.

\bibitem{bonyah2020hopf}
E.~Bonyah, F.~Al~Basir, and S.~Ray, ``Hopf bifurcation in a mathematical model
  of tuberculosis with delay,'' in \emph{Mathematical Modelling, Optimization,
  Analytic and Numerical Solutions}.\hskip 1em plus 0.5em minus 0.4em\relax
  Springer, 2020, pp. 301--311.

\bibitem{brauer2019mathematical}
F.~Brauer, C.~Castillo-Chavez, and Z.~Feng, \emph{Mathematical models in
  epidemiology}.\hskip 1em plus 0.5em minus 0.4em\relax Springer, 2019,
  vol.~32.

\bibitem{castillo2004dynamical}
C.~Castillo-Chavez and B.~Song, ``Dynamical models of tuberculosis and their
  applications,'' \emph{Mathematical Biosciences \& Engineering}, vol.~1,
  no.~2, p. 361, 2004.

\bibitem{chiang2005exogenous}
C.-Y. Chiang and L.~W. Riley, ``Exogenous reinfection in tuberculosis,''
  \emph{The Lancet infectious diseases}, vol.~5, no.~10, pp. 629--636, 2005.

\bibitem{das2025exploring}
S.~Das, T.~Sarkar, and P.~Biswas, ``Exploring the stability of equilibria and
  hopf-bifurcation in a tuberculosis model with delayed treatment,''
  \emph{Physica Scripta}, vol. 100, no.~1, p. 015271, 2025.

\bibitem{das2024exploring}
S.~Das, P.~K. Srivastava, and P.~Biswas, ``Exploring hopf-bifurcations and
  endemic bubbles in a tuberculosis model with behavioral changes and treatment
  saturation,'' \emph{Chaos: An Interdisciplinary Journal of Nonlinear
  Science}, vol.~34, no.~1, p. 013126, 2024.

\bibitem{deiss2009tuberculosis}
R.~G. Deiss, T.~C. Rodwell, and R.~S. Garfein, ``Tuberculosis and illicit drug
  use: review and update,'' \emph{Clinical Infectious Diseases}, vol.~48,
  no.~1, pp. 72--82, 2009.

\bibitem{diekmann2010construction}
O.~Diekmann, J.~A.~P. Heesterbeek, and M.~G. Roberts, ``The construction of
  next-generation matrices for compartmental epidemic models,'' \emph{Journal
  of the royal society interface}, vol.~7, no.~47, pp. 873--885, 2010.

\bibitem{feng2000model}
Z.~Feng, C.~Castillo-Chavez, and A.~F. Capurro, ``A model for tuberculosis with
  exogenous reinfection,'' \emph{Theoretical population biology}, vol.~57,
  no.~3, pp. 235--247, 2000.

\bibitem{gao2018optimal}
D.~Gao and N.-j. Huang, ``Optimal control analysis of a tuberculosis model,''
  \emph{Applied Mathematical Modelling}, vol.~58, pp. 47--64, 2018.

\bibitem{gelman2013bayesian}
A.~Gelman, J.~B. Carlin, H.~S. Stern, D.~B. Dunson, A.~Vehtari, and D.~B.
  Rubin, \emph{Bayesian Data Analysis}, 3rd~ed.\hskip 1em plus 0.5em minus
  0.4em\relax Boca Raton, Florida: CRC, 2013.

\bibitem{gerberry2016practical}
D.~J. Gerberry, ``Practical aspects of backward bifurcation in a mathematical
  model for tuberculosis,'' \emph{Journal of theoretical biology}, vol. 388,
  pp. 15--36, 2016.

\bibitem{Hastings1970}
W.~K. Hastings, ``Monte carlo sampling methods using markov chains and their
  applications,'' \emph{Biometrika}, vol.~57, no.~1, pp. 97--109, 04 1970.

\bibitem{houben2016global}
R.~M. Houben and P.~J. Dodd, ``The global burden of latent tuberculosis
  infection: a re-estimation using mathematical modelling,'' \emph{PLoS
  medicine}, vol.~13, no.~10, p. e1002152, 2016.

\bibitem{kermack1927}
W.~O. Kermack and A.~G. McKendrick, ``A contribution to the mathematical theory
  of epidemics,'' \emph{Proceedings of the royal society of london. Series A,
  Containing papers of a mathematical and physical character}, vol. 115, no.
  772, pp. 700--721, 1927.

\bibitem{kuznetsov2004elements}
Y.~A. Kuznetsov, \emph{Elements of Applied Bifurcation Theory}, 3rd~ed., ser.
  Applied Mathematical Sciences.\hskip 1em plus 0.5em minus 0.4em\relax New
  York: Springer, 2004, vol. 112.

\bibitem{Lawn2011}
S.~D. Lawn and A.~I. Zumla, ``Tuberculosis,'' \emph{The Lancet}, vol. 378, pp.
  57--72, 2011.

\bibitem{liu2011global}
J.~Liu and T.~Zhang, ``Global stability for a tuberculosis model,''
  \emph{Mathematical and Computer Modelling}, vol.~54, no. 1-2, pp. 836--845,
  2011.

\bibitem{liu1994criterion}
W.~M. Liu, ``Criterion of {H}opf bifurcations without using eigenvalues,''
  \emph{Journal of Mathematical Analysis and Applications}, vol. 182, no.~1,
  pp. 250--256, 1994.

\bibitem{lonnroth2010consistent}
K.~L{\"o}nnroth, B.~G. Williams, P.~Cegielski, and C.~Dye, ``A consistent
  log-linear relationship between tuberculosis incidence and body mass index,''
  \emph{International journal of epidemiology}, vol.~39, no.~1, pp. 149--155,
  2010.

\bibitem{ma2018quantifying}
Y.~Ma, C.~Horsburgh, L.~White, and H.~Jenkings, ``Quantifying {TB}
  transmission: a systematic review of reproduction number and serial interval
  estimates for tuberculosis.'' \emph{Epidemiology and Infection}, vol. 146,
  pp. 1478--1494, 2018.

\bibitem{paicanadian}
M.~Pai, J.~Minion, F.~Jamieson, J.~Wolfe, and A.~M. Behr, ``Canadian
  tuberculosis standards 7th edition.''

\bibitem{Peter2025model}
O.~Peter, D.~Aldila, T.~Ayoola, G.~Balogun, and F.~A. Oguntolu, ``Modeling
  tuberculosis dynamics with vaccination and treatment strategies,''
  \emph{Scientific African}, p. e02647, 2025.

\bibitem{rehm2009association}
J.~Rehm, A.~V. Samokhvalov, M.~G. Neuman, R.~Room, C.~Parry, K.~L{\"o}nnroth,
  J.~Patra, V.~Poznyak, and S.~Popova, ``The association between alcohol use,
  alcohol use disorders and tuberculosis (tb). a systematic review,'' \emph{BMC
  public health}, vol.~9, no.~1, pp. 1--12, 2009.

\bibitem{rionero2019hopf}
S.~Rionero, ``Hopf bifurcations in dynamical systems,'' \emph{Ricerche di
  Matematica}, vol.~68, no.~2, pp. 811--840, 2019.

\bibitem{silva2016optimal}
C.~J. Silva, H.~Maurer, and D.~F. Torres, ``Optimal control of a tuberculosis
  model with state and control delays,'' \emph{arXiv preprint
  arXiv:1606.08721}, 2016.

\bibitem{silva2015optimal}
C.~J. Silva and D.~F. Torres, ``Optimal control of tuberculosis: a review,''
  \emph{Dynamics, games and science}, pp. 701--722, 2015.

\bibitem{strogatz2018nonlinear}
S.~H. Strogatz, \emph{Nonlinear Dynamics and Chaos: With Applications to
  Physics, Biology, Chemistry, and Engineering}, 2nd~ed.\hskip 1em plus 0.5em
  minus 0.4em\relax Boca Raton: CRC Press, 2018.

\bibitem{vesga2022prioritising}
J.~F. Vesga, C.~Lienhardt, P.~Nsengiyumva, J.~R. Campbell, O.~Oxlade, S.~den
  Boon, D.~Falzon, K.~Schwartzman, G.~Churchyard, and N.~Arinaminpathy,
  ``Prioritising attributes for tuberculosis preventive treatment regimens: a
  modelling analysis,'' \emph{BMC medicine}, vol.~20, no.~1, pp. 1--12, 2022.

\bibitem{walzl2011immunological}
G.~Walzl, K.~Ronacher, W.~Hanekom, T.~J. Scriba, and A.~Zumla, ``Immunological
  biomarkers of tuberculosis,'' \emph{Nature Reviews Immunology}, vol.~11,
  no.~5, pp. 343--354, 2011.

\bibitem{wangari2018backward}
I.~M. Wangari and L.~Stone, ``Backward bifurcation and hysteresis in models of
  recurrent tuberculosis,'' \emph{PloS one}, vol.~13, no.~3, p. e0194256, 2018.

\bibitem{world2015end}
WHO \emph{et~al.}, ``The end {TB} strategy,'' World Health Organization, Tech.
  Rep., 2015.

\bibitem{WHOTuberculosis_DrugResistance}
------, ``{Drug Resistance: tuberculosis},''
  https://www.who.int/drugresistance/tb/en, 2021.

\bibitem{WHOTuberculosis2024}
------, ``{Global tuberculosis report 2024. World Health Organization},''
  https://www.who.int/teams/global-tuberculosis-programme/tb-reports/global-tuberculosis-report-2024,
  2024, {Accessed}: 2025-24-02.

\bibitem{WHOTuberculosis2025}
------, ``{Global tuberculosis report 2025. World Health Organization},''
  https://www.who.int/teams/global-programme-on-tuberculosis-and-lung-health/tb-reports/global-tuberculosis-report-2025,
  2024, {Accessed}: 2025-24-02.

\bibitem{zhang2023dynamical}
Z.~Zhang, W.~Zhang, K.~S. Nisar, N.~Gul, A.~Zeb, and V.~Vijayakumar,
  ``Dynamical aspects of a tuberculosis transmission model incorporating
  vaccination and time delay,'' \emph{Alexandria Engineering Journal}, vol.~66,
  pp. 287--300, 2023.





\end{thebibliography}
\renewcommand{\refname}{}
\begingroup
\setlength{\itemsep}{0pt}\setlength{\parskip}{0pt}

\endgroup
\end{document}